\documentclass[11pt,reqno]{amsart}

\newtheorem{thm}{Theorem}[section]
\newtheorem*{thm*}{Theorem}
\newtheorem{lem}[thm]{Lemma}
\newtheorem*{lem*}{Lemma}

\newtheorem{prop}[thm]{Proposition}

\theoremstyle{definition}

\newtheorem{assump}[thm]{Assumption}
\newtheorem{case}{Case}\renewcommand{\thecase}{}
\newtheorem*{case*}{Case}

\newtheorem{defn}[thm]{Definition}
\newtheorem*{defn*}{Definition}

\newtheorem*{exmp*}{Example}

\newtheorem{rmk}[thm]{Remark}
\newtheorem*{rmk*}{Remark}
\newtheorem{step}{Step}\renewcommand{\thestep}{}

\theoremstyle{remark}

\makeatletter
\def\alphenumi{
  \def\theenumi{\alph{enumi}}
  \def\p@enumi{\theenumi}
  \def\labelenumi{(\@alph\c@enumi)}}
\makeatother

\makeatletter
\def\thecase{\@arabic\c@case}
\makeatother

\makeatletter
\def\thestep{\@arabic\c@step}
\makeatother

\newcount\hh
\newcount\mm
\mm=\time
\hh=\time
\divide\hh by 60
\divide\mm by 60
\multiply\mm by 60
\mm=-\mm
\advance\mm by \time
\def\hhmm{\number\hh:\ifnum\mm<10{}0\fi\number\mm}

\let\oldmarginpar\marginpar
\renewcommand\marginpar[1]{\-\oldmarginpar[\raggedleft\footnotesize #1]%
{\raggedright\footnotesize #1}}

\newcommand\dotprod{\hbox{$\cdot$}}

\newcommand\barU{{\overline{U}}}

\newcommand\EE{\mathbb{E}}

\newcommand\NN{\mathbb{N}}

\newcommand\PP{\mathbb{P}}
\newcommand\QQ{\mathbb{Q}}
\newcommand\RR{\mathbb{R}}

\newcommand\cF{{\mathcal{F}}}
\newcommand\cG{{\mathcal{G}}}

\newcommand\cT{{\mathcal{T}}}

\newcommand\eps{\varepsilon}

\numberwithin{equation}{section}

\theoremstyle{default}
\newtheorem{locclaim}{Claim}

\usepackage{amssymb}
\usepackage{hyperref}
\usepackage{mathrsfs}
\usepackage{slashed}
\usepackage[usenames]{color}
\usepackage{url}
\usepackage{verbatim}

\allowdisplaybreaks

\hypersetup{pdftitle={Nonlocal operators}}

\hypersetup{pdfauthor={Donatella Danielli, Arshak Petrosyan, and Camelia A. Pop}}

\begin{document}

\title[Obstacle problems for nonlocal operators]{Obstacle problems for nonlocal operators}

\author[D. Danielli]{Donatella Danielli}
\address[DD]{Department of Mathematics, Purdue University, West Lafayette, IN 47907}
\email{danielli@math.purdue.edu}

\author[A. Petrosyan]{Arshak Petrosyan}
\address[AP]{Department of Mathematics, Purdue University, West Lafayette, IN 47907}
\email{arshak@math.purdue.edu}

\author[C. A. Pop]{Camelia A. Pop}
\address[CP]{School of Mathematics, University of Minnesota, Vincent Hall, 206 Church St. SE, Minneapolis, MN 55455}
\email{capop@umn.edu}

\begin{abstract}
We prove existence, uniqueness, and regularity of viscosity solutions to the stationary and evolution obstacle problems defined by a class of nonlocal operators that are not stable-like and may have supercritical drift. We give sufficient conditions on the coefficients of the operator to obtain H\"older and Lipschitz continuous solutions. The class of nonlocal operators that we consider include non-Gaussian asset price models widely used in mathematical finance, such as Variance Gamma Processes and Regular L\'evy Processes of Exponential type. In this context, the viscosity solutions that we analyze coincide with the prices of perpetual and finite expiry American options.
\end{abstract}

\date{\today{ }\hhmm}

\maketitle


\section{Introduction}
\label{sec:Introduction}

The nonlocal operators that we consider are the infinitesimal generators of strong Markov processes that are solutions to stochastic equations of the form:
\begin{equation}
\label{eq:Process}
dX(t) = b(X(t-))\, dt + \int_{\RR^n\setminus\{O\}} F((X(t-),y) \, \widetilde N(dt, dy),\quad\forall\, t>0,
\end{equation}
where $O$ denotes the origin in $\RR^n$, $\widetilde N(dt, dy)$ is a compensated Poisson random measure with L\'evy measure $\nu(dy)$, and the coefficients $b(x)$ and $F(x,y)$ appearing in identity \eqref{eq:Process} are assumed to satisfy:

\begin{assump}[Coefficients]
\label{assump:Coeff}
There is a positive constant $K$ such that:
\begin{enumerate}
\item[1.]
For all $x, x_1, x_2\in\RR^n$, we have that
\begin{align}
\label{eq:F_Lipschitz_squared}
&\int_{\RR^n\setminus\{O\}} |F(x_1, y)-F(x_2, y)|^2 \, d\nu(y) \leq K |x_1-x_2|^2,\\
\label{eq:F_Levy_cond}
&\sup_{z\in B_{|y|}}|F(x,z)| \leq \rho(y),\quad\forall\, x,y\in\RR^n,
\quad\hbox{ and }\quad
\int_{\RR^n\setminus\{O\}} \left(|y|\vee\rho(y)\right)^2 \,\nu(dy) \leq K,
\end{align}
where $\rho:\RR^n\rightarrow[0,\infty)$ is a measurable function.
\item[2.]
The coefficient $b:\RR^n\rightarrow\RR^n$ is bounded and Lipschitz continuous, i.e., $b\in C^{0,1}(\RR^n)$.
\end{enumerate}
\end{assump}

Let $(\Omega,\cF,\{\cF_t\}_{t\geq0},\PP)$ be a filtered probability space that satisfies the usual hypotheses, \cite[\S~I.1]{Protter}, and that supports a Poisson random measure, $N(dt,dy)$, with L\'evy measure, $\nu(dy)$. 
Assumption~\ref{assump:Coeff} and \cite[Theorem~V.3.6]{Protter} ensure that, for any initial condition $X^x(0)=x\in\RR^n$, the stochastic equation \eqref{eq:Process} admits a unique strong solution $\{X^x(t)\}_{t\geq 0}$ with right-continuous and left-limits (RCLL) paths a.s.. Moreover, from \cite[Theorem~V.6.32]{Protter}\footnote{\cite[Theorem~V.6.32]{Protter} applies to the case when $b\equiv 0$ in the stochastic equation \eqref{eq:Process}, but it is not difficult to see that the proof immediately extends to the case of Lipschitz continuous drift coefficients $b(x)$, as we suppose in Assumption~\ref{assump:Coeff}.}
it follows that the process $\{X^x(t)\}_{t\geq 0}$ satisfies the strong Markov property. Thus, the Markov process $\{X^x(t)\}_{t\geq 0}$ is completely characterized by its infinitesimal generator, which is given by the nonlocal operator:
\begin{equation}
\label{eq:Operator}
L u(x) := b(x) \dotprod \nabla u(x) + \int_{\RR^n\setminus\{O\}}\left(u(x+F(x,y)) - u(x) - \nabla u(x)\dotprod F(x,y)\right)\,\nu(dy),
\end{equation}
for all $u\in C^{\infty}_c(\RR^n)$. Our goal is to study the existence, uniqueness, and regularity properties of viscosity solutions to the stationary and evolution obstacle problems associated to the operator $L$.

\subsection{Stationary obstacle problem}
\label{subsec:Stationary}
In this section we state our results related to the existence, uniqueness, and regularity of viscosity solutions to the stationary obstacle problem defined by the nonlocal operator $L$,
\begin{equation}
\label{eq:Stationary_obs_problem}
\min\{-Lv + c v - f, v - \varphi\} = 0,\quad\hbox{ on }\RR^n,
\end{equation}
where $c:\RR^n\rightarrow\RR$ is the zeroth order term,
$f:\RR^n\rightarrow\RR$ is the source function, and
$\varphi:\RR^n\rightarrow\RR$ is the obstacle function. For the
existence results in the stationary case we will also assume that the jump size $F(x,y)$ does not depend on the state variable, that is,
\begin{equation}
\label{eq:Jump_size}
F(x,y) = y,\quad\forall\, x, y\in\RR^n.
\end{equation}
Let $\cT$ denote the set of $\PP$-a.s.\ finite stopping times adapted to the filtration $\{\cF_t\}_{t\geq 0}$. Solutions to the obstacle problem \eqref{eq:Stationary_obs_problem} are constructed using the stochastic representation formula:
\begin{equation}
\label{eq:Value_function_stat}
v(x) := \sup\{v(x;\tau):\,\tau\in\cT\},
\end{equation}
where we denote
\begin{equation}
\label{eq:Value_function_stat_aux}
v(x;\tau) :=\EE \left[e^{-\int_0^{\tau} c(X^x(s))\, ds} \varphi(X^x(\tau))
+ \int_0^{\tau} e^{-\int_0^ t c(X^x(s))\, ds} f(X^x(t))\, dt\right],
\end{equation}
where $\{X^x(t)\}_{t\geq 0}$ is the unique solution to the stochastic equation \eqref{eq:Process} with initial condition $X^x(0)=x$,
for all $x\in\RR^n$. We first state

\begin{prop}[Regularity of the value function]
\label{prop:Regularity_stat}
Suppose that Assumption~\ref{assump:Coeff} and condition \eqref{eq:Jump_size} hold. Let $c,\varphi,f:\RR^n\rightarrow \RR$ be bounded Lipschitz continuous functions, and assume that there is a positive constant, $c_0$, with the property that
\begin{equation}
\label{eq:c_lower_bound}
c(x) \geq c_0>0,\quad\forall\, x\in\RR^n.
\end{equation}
Then the following hold:
\begin{enumerate}
\item[(i)] \emph{(H\"older continuity)}
There is a constant, $\alpha = \alpha([b]_{C^{0,1}(\RR^n)}, c_0) \in (0,1)$, such that the value function $v$ defined in \eqref{eq:Value_function_stat} belongs to $C^{\alpha}(\RR^n)$. 
\item[(ii)] \emph{(Lipschitz continuity)}
If in addition we have that 
\begin{equation}
\label{eq:Cond_for_Lipschitz_stat}
c_0\geq [b]_{C^{0,1}(\RR^n)},
\end{equation}
then the value function $v$ in \eqref{eq:Value_function_stat} belongs to $C^{0,1}(\RR^n)$.
\end{enumerate}
\end{prop}

\begin{defn}[Viscosity solutions]
\label{defn:Viscosity_sol_stat}
Let $v\in C(\RR^n)$. We say that $v$ is a viscosity subsolution (supersolution) to the stationary obstacle problem 
\eqref{eq:Stationary_obs_problem} if, for all $u\in C^2(\RR^n)$ such that $v-u$ has a global max (min) at $x_0\in\RR^n$ and 
$u(x_0)=v(x_0)$, then
\begin{equation}
\label{eq:Sub_super_sol_stat}
\min\{-Lu(x_0) + c(x_0) u(x_0) - f(x_0), u(x_0) - \varphi(x_0)\} \leq (\geq)\, 0.
\end{equation}
We say that $v$ is a viscosity solution to equation \eqref{eq:Stationary_obs_problem} if it is both a sub- and supersolution.
\end{defn}

\begin{thm}[Existence of viscosity solution]
\label{thm:Existence_stat}
Suppose that the hypotheses of Proposition~\ref{prop:Regularity_stat} hold, and that
\begin{equation}
\label{eq:Jump_mart_stat}
\int_{\RR^n\setminus\{O\}} |y|^{2\alpha}\, \nu(dy) <\infty
\end{equation}
where $\alpha\in (0,1)$ is the constant appearing in Proposition~\ref{prop:Regularity_stat} (i).
Then the value function $v$ defined in \eqref{eq:Value_function_stat} is a viscosity solution to the stationary obstacle problem \eqref{eq:Stationary_obs_problem}.
\end{thm}

\begin{thm}[Uniqueness of viscosity solution]
\label{thm:Uniqueness_stat}
Suppose that Assumption~\ref{assump:Coeff} holds, that $c,f,\varphi$ belong to $C(\RR^n)$, and $c$ satisfies condition \eqref{eq:c_lower_bound}. If the stationary obstacle problem \eqref{eq:Stationary_obs_problem} has a viscosity solution, then it is unique.
\end{thm}

\begin{rmk}[Condition \eqref{eq:c_lower_bound} on the zeroth order term $c(x)$]
\label{rmk:c_lower_bound_infty}
Condition \eqref{eq:c_lower_bound} in Theorem \ref{thm:Uniqueness_stat} can be replaced by the less restrictive assumption that $c(x)$ is a positive function on $\RR^n$ and 
\begin{equation}
\label{eq:c_lower_bound_infty}
\limsup_{|x|\to\infty}\frac{1}{|x|c(x)} = 0.
\end{equation}
\end{rmk}

\subsection{Evolution obstacle problem}
\label{subsec:Evolution}
We next consider the questions of existence, uniqueness, and regularity of viscosity solutions to the evolution obstacle problem defined by the nonlocal operator $L$,
\begin{equation}
\label{eq:Evolution_obs_problem}
\begin{aligned}
\left\{
\begin{array}{rll}
\min\{-v_t-Lv + c v - f, v - \varphi\} &= 0,& \quad\hbox{ on } [0,T)\times\RR^n,\\
v(T,\cdot) &= g,&\quad\hbox{ on } \RR^n,
\end{array}
\right.
\end{aligned}
\end{equation}
where we assume the compatibility assumption 
\begin{equation}
\label{eq:Compatibility_evol}
g\geq \varphi(T,\cdot)\quad\hbox{ on }\quad \RR^n.
\end{equation}
Let $\cT_t$ denote the set of stopping times $\tau\in\cT$ bounded by $t$, for all $t\geq 0$. Solutions to problem \eqref{eq:Evolution_obs_problem} are constructed using the stochastic representation formula,
\begin{equation}
\label{eq:Value_function_evol}
v(t,x) := \sup \{ v(t,x;\tau):\, \tau\in\cT_{T-t}\},
\end{equation}
where we define
\begin{equation}
\label{eq:Value_function_evol_aux}
\begin{aligned}
v(t,x;\tau) &:= \EE\left[e^{-\int_0^{\tau} c(t+s, X^x(s))\, ds}\varphi(t+\tau, X^x(\tau)) \mathbf{1}_{\{\tau<T-t\}}\right]\\
&\quad+ \EE \left[e^{-\int_0^{\tau} c(t+s, X^x(s))\, ds}g(X^x(T-t)) \mathbf{1}_{\{\tau = T-t\}}\right]\\
&\quad+ \EE \left[\int_0^{\tau} e^{-\int_0^s c(t+r,X^x(r))\, dr} f(t+s,X^x(s))\, ds\right],
\end{aligned}
\end{equation}
for all $(t,x)\in[0,T]\times\in\RR^n$. The value function $v(t,x)$ satisfies:

\begin{prop}[Regularity of the value function]
\label{prop:Regularity_evol}
In addition to Assumption~\ref{assump:Coeff} suppose that $c,\varphi,f$ belong to $C^{0,1}([0,T]\times\RR^n)$, the final condition $g$ is in $C^{0,1}(\RR^n)$, and the compatibility condition \eqref{eq:Compatibility_evol} holds. Then the value function $v$ defined in \eqref{eq:Value_function_evol} belongs to $C^{\frac{1}{2}}_tC^{0,1}_x([0,T]\times\RR^n)$.
\end{prop}

We next define a notion of viscosity solution for the evolution obstacle problem \eqref{eq:Evolution_obs_problem} extending that of its stationary analogue for equation \eqref{eq:Stationary_obs_problem} similarly to the ideas described in \cite[\S~8]{Crandall_Ishii_Lions_1992}:

\begin{defn}[Viscosity solutions]
\label{defn:Viscosity_sol_evol}
Let $v\in C(\RR^n)$. We say that $v$ is a viscosity subsolution (supersolution) to the evolution obstacle problem 
\eqref{eq:Evolution_obs_problem} if 
\begin{equation}
\label{eq:Final_cond_evol}
v(T,\cdot) \leq (\geq) g,
\end{equation}
and, for all $u\in C^1_tC^2_x([0,T]\times\RR^n)$ such that $u-v$ has a global max (min) at $(t_0,x_0)\in[0,T)\times\RR^n$ and $u(t_0,x_0)=v(t_0,x_0)$, we have that
\begin{equation}
\label{eq:Sub_super_sol_evol}
\min\{-u_t(t_0,x_0)-Lu(t_0,x_0) + c(t_0,x_0) u(t_0,x_0) - f(t_0,x_0), u(t_0,x_0) - \varphi(t_0,x_0)\} \leq (\geq)\, 0.
\end{equation}
We say that $v$ is a viscosity solution to equation \eqref{eq:Evolution_obs_problem} if it is both a sub- and supersolution.
\end{defn}

\begin{thm}[Existence of viscosity solution]
\label{thm:Existence_evol}
Suppose that the hypotheses of Proposition~\ref{prop:Regularity_evol} hold. Then the value function $v$ defined in \eqref{eq:Value_function_evol} is a viscosity solution to the evolution obstacle problem \eqref{eq:Evolution_obs_problem}.
\end{thm}

\begin{rmk}[Assumptions in evolution vs stationary cases]
We note that in the case of the evolution obstacle problem we allow the jump size $F(x,y)$ to depend on the spatial variable $x$, in contrast to assumption \eqref{eq:Jump_size} in the case of the stationary obstacle problem.

We also note that we do not require condition \eqref{eq:Jump_mart_stat} to
hold in the statement of Theorem~\ref{thm:Existence_evol}. We are able
to remove this condition because in the evolution case Proposition~\ref{prop:Regularity_evol} shows that the value function is Lipschitz
continuous in the spatial variable, as opposed to the stationary case
where we prove in Proposition~\ref{prop:Regularity_stat} that the
value function is $\alpha$-H\"older continuous.
\end{rmk}

\begin{thm}[Uniqueness of viscosity solution]
\label{thm:Uniqueness_evol}
Suppose that Assumption~\ref{assump:Coeff} is satisfied, $g$ belongs to $C(\RR^n)$, $c,f,\varphi$ are in $C([0,T]\times\RR^n)$, the compatibility condition \eqref{eq:Compatibility_evol} holds, and 
\begin{equation}
\label{eq:Interior_point}
\lim_{y\to O} F(x,y) = 0,\quad\forall\, x\in\RR^n.
\end{equation}
If the obstacle problem \eqref{eq:Evolution_obs_problem} has a solution, then it is unique.
\end{thm}

\subsection{Applications to mathematical finance}
\label{sec:Applications}
In mathematical finance stochastic representations of the form \eqref{eq:Value_function_stat} and \eqref{eq:Value_function_evol} have the meaning of the prices of American perpetual and finite expiry options, respectively. To make this correspondence, in the evolution obstacle problem \eqref{eq:Evolution_obs_problem}, we set $g\equiv \varphi$, $f\equiv 0$, and we choose the obstacle function $\varphi$ to depend only on the spatial variable $x$, and to coincide with the payoff of the American option. In addition, we assume that $n=1$, the zeroth order term $c\equiv r>0$, where $r$ is the risk-free interest rate, and the asset price process can be written in the form $S(t)=e^{X(t)}$, where $\{X(t)\}_{t\geq 0}$ solves the stochastic equation \eqref{eq:Process}. Most importantly, we need to ensure that the discounted asset price process $\{e^{-rt} S(t)\}_{t\geq 0}$ is a martingale in order to obtain an arbitrage-free market. Because the markets containing asset prices driven by discontinuous L\'evy processes that are not Poisson processes are incomplete, the motivation to choose to price options using the risk-free probability measure given by the distribution of the asset price requires careful thought. However, we do not address this problem in our paper, but see \cite[\S~1.3.4]{Boyarchenko_Levendorskii_2002b} and 
\cite[Chapter~9]{ContTankov} for more discussions on this problem.

Assume that $\{X(t)\}_{t\geq 0}$ is a one-dimensional L\'evy process that satisfies the stochastic equation:
\begin{equation}
\label{eq:Process_const_drift}
dX(t) = b\, dt + \int_{\RR^n} y \, \widetilde N(dt, dy),\quad\forall\, t>0,
\end{equation}
where $b$ is a real constant and $\widetilde N(dt, dy)$ is a compensated Poisson random measure with L\'evy measure $\nu(dy)$. Using \cite[Theorem~5.2.4 and Corollary~5.2.2]{Applebaum} a sufficient 
condition that guarantees that the discounted asset price process $\{e^{-rt+X(t)}\}_{t\geq 0}$ is a martingale is:
\begin{equation}
\label{eq:Martingale_cond}
\begin{aligned}
\int_{|x|\geq 1} e^x\, \nu(dx) &<\infty,\\
-r+\psi(-i) &= 0,
\end{aligned}
\end{equation} 
where $\psi(\xi)$ denotes the characteristic exponent of the L\'evy process $\{X(t)\}_{t\geq 0}$, that is,
\begin{equation}
\label{eq:Characteristic_exponent}
\psi(\xi) = ib\xi+\int_{\RR\setminus\{0\}} (e^{ix\xi}-1-ix\xi)\,\nu(dx).
\end{equation} 
Examples in mathematical finance to which our results applies include the Variance Gamma Process \cite{Madan_Seneta_1990} and Regular L\'evy Processes of Exponential type (RLPE) \cite{Boyarchenko_Levendorskii_2002b}. 

\subsubsection{Variance Gamma Process}
\label{sec:VG}
Following \cite[Identity (6)]{Carr_Geman_Madan_Yor_2002}, the Variance Gamma Process $\{ X(t)\}_{t\geq 0}$ with parameters $\nu,\sigma,$ and $\theta$ has L\'evy measure given by
\begin{align*}
\nu(dx) = \frac{1}{\nu|x|}\left(e^{-\frac{|x|}{\eta_p}}\mathbf{1}_{\{x>0\}} + e^{-\frac{|x|}{\eta_n}}\mathbf{1}_{\{x<0\}}\right)\, dx,
\end{align*}
where $\eta_p>\eta_n$ are the roots of the equation $x^2-\theta\nu x-\sigma^2\nu/2=0$, and $\nu,\sigma, \theta$ are positive constants. From \cite[Identity (4)]{Carr_Geman_Madan_Yor_2002}, we have that the characteristic exponent of the Variance Gamma Process with constant drift $b\in\RR$, $\{X(t)+bt\}_{t\geq 0}$, has the expression:
$$
\psi_{\hbox{\tiny{VG}}}(\xi) = \frac{1}{\nu}\ln \left(1-i\theta\nu\xi + \frac{1}{2}\sigma^2\nu\xi^2\right) + ib\xi,
\quad\forall\, \xi\in\RR,
$$
and so the infinitesimal generator of $\{X(t)+bt\}_{t\geq 0}$ is given by
$$
L = \frac{1}{\nu}\ln(1-\theta\nu\nabla - \frac{1}{2}\sigma^2\Delta)+b\dotprod\nabla,
$$
which is a sum of a pseudo-differential operator of order less that any $s>0$ and one of order $1$. When $\eta_p<1$ and $r=\psi_{VG}(-i)$\footnote{The fact that $r>0$ and $r=\psi_{VG}(-i)$ implies that $1+\theta\nu - \frac{1}{2}\sigma^2\nu$ is a positive constant, and the drift $b$ satisfies the inequality $b>-\frac{1}{\nu}\ln \left(1+\theta\nu - \frac{1}{2}\sigma^2\nu\right)$.}, condition \eqref{eq:Martingale_cond} is satisfied and the discounted asset price process $\{e^{-rt+X(t)}\}_{t\geq 0}$ is a martingale. Thus, applying the results in \S~\ref{subsec:Stationary} and \S~\ref{subsec:Evolution} to the Variance Gamma Process $\{X(t)\}_{t\geq 0}$ with constant drift $b$, we obtain that the prices of perpetual and finite expiry American options with bounded and Lipschitz payoffs are Lipschitz functions in the spatial variable. Given that the nonlocal component of the infinitesimal generator $L$ has order less than any $s>0$, this may be the optimal regularity of solutions that we can expect.

\subsubsection{Regular L\'evy Processes of Exponential type}
\label{sec:RLPE}
Following \cite[Chapter~3]{Boyarchenko_Levendorskii_2002b}, for parameters $\lambda_{-} < 0 < \lambda_{+}$, a L\'evy process is said to be of exponential type $[\lambda_{-}, \lambda_{+}]$ if it has a L\'evy measure $\nu(dx)$ such that 
$$
\int_{-\infty}^{-1} e^{-\lambda_{+}x} \nu(dx) + \int_1^{\infty} e^{-\lambda_{-}x}\nu(dx) <\infty.
$$ 
A non-Gaussian L\'evy process is said to be a Regular L\'evy Processes of Exponential type $[\lambda_{-}, \lambda_{+}]$ and order $\nu$ if it has exponential type $[\lambda_{-}, \lambda_{+}]$ and, in a neighborhood of zero, the L\'evy measure can be represented as $\nu(dx) = f(x)\,dx$, where $f(x)$ satisfies the property that
$$
|f(x) - c|x|^{-\nu-1}| \leq C|x|^{-\nu'-1},\quad\forall\, |x|\leq 1,
$$
for constants $\nu'<\nu$, $c>0$, and $C>0$. Our results apply to RLPE type $[\lambda_{-}, \lambda_{+}]$, when we choose the parameters $\lambda_-\leq-1$ and $\lambda_+ \geq 1$\footnote{See the first identity in condition \eqref{eq:Martingale_cond}.}.

The class of RLPE include the CGMY/KoBoL processes introduced in \cite{Carr_Geman_Madan_Yor_2002}. Following \cite[Equation (7)]{Carr_Geman_Madan_Yor_2002}, CGMY/KoBoL processes are characterized by a L\'evy measure of the form
$$
\nu(dx) = \frac{C}{|x|^{1+Y}}\left(e^{-G|x|}\mathbf{1}_{\{x<0\}} + e^{-M|x|}\mathbf{1}_{\{x>0\}}\right)\, dx,
$$
where the parameters $C>0$, $G, M \geq 0$, and $Y<2$. Our results apply to CGMY/KoBoL processes, when we choose the parameter $M>1$ and $Y<2$, or $M=1$ and $0<Y<2$\footnote{See the first identity in condition \eqref{eq:Martingale_cond}.}.

We remark that a sufficient condition on the L\'evy measure to ensure that perpetual American put option prices are Lipschitz continuous, but not continuously differentiable, is provided in \cite[Theorem~5.4, p.~133]{Boyarchenko_Levendorskii_2002b}. However, the condition is in terms of the Wiener-Hopf factorization for the characteristic exponent of the L\'evy process, and it is difficult to find a concrete example for which it holds.

\subsection{Comparison with previous research}
\label{subsec:Comparison}
In \cite{Mordecki_2002}, the author establishes closed-form formulas for prices of perpetual American call and put options on a stock driven by a general L\'evy process, in terms of the distribution of the supremum and the infimum of the process, respectively. In \cite{Boyarchenko_Levendorskii_2002, Boyarchenko_Levendorskii_2002b}, in the framework of Regular L\'evy Processes of Exponential type, the authors obtain closed-form formulas for prices of perpetual American call and put options via the Wiener-Hopf factorization method distribution of the supremum and the infimum of the process. Compared with \cite{Mordecki_2002, Boyarchenko_Levendorskii_2002, Boyarchenko_Levendorskii_2002b}, in our work we allow more general payoff functions for which we study both the perpetual and the finite expiry American options, together with the regularity properties of the option prices. Our results apply to multi-dimensional Markov processes that may not be L\'evy processes, but when restricted to the class of L\'evy processes, 
we allow both a more restrictive family than in \cite{Mordecki_2002}, but also more general than in \cite{Boyarchenko_Levendorskii_2002, Boyarchenko_Levendorskii_2002b}.

The nonlocal operators most often studied in the context of obstacle problems are stable-like \cite{Caffarelli_Ros-Oton_Serra_2016}. However, the nonlocal operators often arising in applications in mathematical finance are not of this form, and in our work we include operators relevant in this field as we described in \S~\ref{sec:Applications}. Their analytic properties appear to be quite different, as the case of Variance Gamma Processes in \S~\ref{sec:VG} shows, and we prove regularity properties of solutions using probabilistic and viscosity solutions arguments. 
The Lipschitz regularity of solutions that we establish in Theorem~\ref{thm:Existence_stat} is optimal for a subclass of  nonlocal operators, as \cite[Theorem~5.4, p.~133]{Boyarchenko_Levendorskii_2002b} proves.

\subsection{Structure of the paper}
\label{sec:Structure}
We prove the main results stated in the introduction in \S~\ref{sec:Stationary} and \S~\ref{sec:Evolution}, respectively. In addition to these results, we first prove in the stationary case a Dynamic Programming Principle and a Comparison Principle in Lemma~\ref{lem:DPP_stat} and Theorem~\ref{thm:Comparison_princ_stat}, respectively. The Dynamic Programming Principle is used in the proof of Theorem~\ref{thm:Existence_stat} where we establish the existence of viscosity solutions, while the Comparison Principle is used in the proof of Theorem~\ref{thm:Uniqueness_stat} to establish the uniqueness of solutions to the stationary obstacle problem \eqref{eq:Stationary_obs_problem}. Analogous results are obtained for the evolution case in Lemma~\ref{lem:DPP_evol} and Theorem~\ref{thm:Comparison_princ_evol}, respectively. In \S~\ref{sec:Notation} we describe the notations and conventions we use throughout the paper.

\subsection{Notation}
\label{sec:Notation}
Let $U\subseteq \RR^n$ be a set. We denote by $C(\bar U)$ the space of continuous functions $u:\bar U\rightarrow\RR^n$, such that
\begin{equation*}
\|u\|_{C(\bar U)} :=\sup_{x\in\bar U} |u(x)| < \infty.
\end{equation*}
We let $C^{\infty}_c(\RR^n)$ be the space of smooth functions defined on $\RR^n$ with bounded derivatives of all orders.
For all $\alpha\in (0,1]$, a function $u:\bar U\rightarrow\RR^n$ belongs to $C^{0,\alpha}(\bar U)$ if 
\begin{align*}
\|u\|_{C^{0,\alpha}(\bar U)} := \|u\|_{C(\bar U)} + [u]_{C^{0,\alpha}(U)}<\infty,
\end{align*}
where, as usual, we define
\begin{equation*}
[u]_{C^{0,\alpha}(U)}:= \sup_{x_1,x_2\in\bar U, x_1\neq x_2} \frac{|u(x_1)-u(x_2)|}{|x_1-x_2|^{\alpha}}.
\end{equation*}
When $\alpha\in (0,1)$, we denote for brevity $C^{\alpha}(\barU) := C^{0,\alpha}(\bar U)$. For all $T>0$, we denote by $C^{\frac{1}{2}}_tC^{0,1}_x([0,T]\times\RR^n)$ the space of functions $u:[0,T]\times\RR^n\rightarrow\RR$ such that
\begin{equation*}
\|u\|_{C^{\frac{1}{2}}_tC^{0,1}_x([0,T]\times\RR^n)}:= \|u\|_{C([0,T]\times\RR^n)}
+
\sup_{\stackrel{t_1,t_2\in [0,T], t_1\neq t_2}{x_1,x_2\in\RR^n, x_1\neq x_2}} 
\frac{|u(t_1,x_1)-u(t_2,x_2)|}{|t_1-t_2|^{\frac{1}{2}}+|x_1-x_2|} <\infty,
\end{equation*}
and we let $C^1_tC^2_x([0,T]\times\RR^n)$ denote the space of functions $u:[0,T]\times\RR^n\rightarrow\RR$ such that the first order derivative in the time variable and the second order derivatives in the spatial variables are continuous and bounded.

Given a set $U\subseteq \RR^n$, we denote by $U^c$ its complement. For all $a,b\in\RR$, we denote
$$
a\wedge b:=\min\{a,b\}
\quad\hbox{ and }\quad
a\vee b:= \max\{a,b\}.
$$
For all $r>0$ and $x\in\RR^n$, we denote by $B_r(x)$ the Euclidean ball of radius $r$ centered at $x$.

\section{Stationary obstacle problem}
\label{sec:Stationary}
In this section, we give the proofs of Proposition~\ref{prop:Regularity_stat}, and Theorems~\ref{thm:Existence_stat} and \ref{thm:Uniqueness_stat}. In addition, we prove a Dynamical Programming Principle in Lemma~\ref{lem:DPP_stat} and a comparison principle in Theorem~\ref{thm:Comparison_princ_stat}. We begin with:

\begin{proof}[Proof of Proposition~\ref{prop:Regularity_stat}]
We denote by $\{X^x(t)\}_{t\geq 0}$ the unique strong solution to the stochastic equation \eqref{eq:Process} with initial condition $X(0)=x\in\RR^n$. Because the functions $\varphi$ and $f$ are bounded and the zeroth order term $c$ satisfies property \eqref{eq:c_lower_bound}, it is clear that the value function $v$ defined in \eqref{eq:Value_function_stat} is bounded. To prove the H\"older continuity of $v$, we use the fact that  
\begin{equation}
\label{eq:Ineq_difference_stat}
|v(x_1)-v(x_2)| \leq \sup_{\tau\in\cT} |v(x_1;\tau) - v(x_2;\tau)|,\quad\forall\, x_1, x_2\in\RR^n,
\end{equation}
and we can assume without loss of generality that $|x_1-x_2|<1$ since $v$ is bounded. Let $T>0$ be a constant. Using definition \eqref{eq:Value_function_stat_aux} of the function $v(x;\tau)$ and condition \eqref{eq:c_lower_bound}, we see that
\begin{align*}
&|v(x_1;\tau) - v(x_2;\tau)|\\ 
&\qquad\leq
\EE \left[e^{-\int_0^{\tau} c(X^{x_1}(s))\, ds} \left|\varphi(X^{x_1}(\tau))-\varphi(X^{x_2}(\tau))\right|
\mathbf{1}_{\{\tau\leq T\}}\right]\\
&\qquad\quad
+ \EE \left[\left|e^{-\int_0^{\tau} c(X^{x_1}(s))\, ds} - e^{-\int_0^{\tau} c(X^{x_2}(s))\, ds}\right| 
|\varphi(X^{x_2}(\tau))|\mathbf{1}_{\{\tau\leq T\}}\right]\\
&\qquad\quad
+\EE \left[\int_0^{\tau\wedge T} e^{-\int_0^t c(X^{x_1}(s))\, ds} \left|f(X^{x_1}(t))-f(X^{x_2}(t)\right|\,dt\right]\\
&\qquad\quad
+ \EE \left[\int_0^{\tau\wedge T}\left|e^{-\int_0^t c(X^{x_1}(s))\, ds} - e^{-\int_0^t c(X^{x_2}(s))\, ds}\right| 
|f(X^{x_2}(t))|\,dt\right]\\
&\qquad\quad
+2\left(\|\varphi\|_{C(\RR^n)} + \|f\|_{C(\RR^n)}\right) e^{-c_0T}.
\end{align*}
Property \eqref{eq:c_lower_bound} and the fact that the functions $c$, $f$, and $\varphi$ are Lipschitz continuous give us that there is a positive constant, $C=C(c_0,[c]_{C^{0,1}(\RR^n)}, \|\varphi\|_{C^{0,1}(\RR^n)},\|f\|_{C^{0,1}(\RR^n)})$, such that
\begin{equation}
\label{eq:Ineq_tau_stat}
\begin{aligned}
&|v(x_1;\tau) - v(x_2;\tau)| \\
&\qquad\leq 
C\EE \left[e^{-c_0(\tau\wedge T)}|X^{x_1}(\tau\wedge T) - X^{x_2}(\tau\wedge T)|\right]\\ 
&\qquad\quad
+ C\EE \left[\int_0^{\tau\wedge T} e^{-c_0t}|X^{x_1}(t) - X^{x_2}(t)|\,dt\right]\\
&\qquad\quad
+C\EE \left[\int_0^{\tau\wedge T} e^{-c_0t} \int_0^t |X^{x_1}(s) - X^{x_2}(s)|\, ds\,dt\right]
+C e^{-c_0 T}.
\end{aligned}
\end{equation}
Using assumption \eqref{eq:Jump_size} in the stochastic equation \eqref{eq:Process}, it follows that
$$
X^{x_i}(t) = x_i + \int_0^t b(X^{x_i}(s))\,ds + \int_0^t\int_{\RR^n\setminus\{O\}} y\,d\widetilde N(ds,dy),
\quad\hbox{ for } i = 1,2,\quad\forall\, t>0,
$$
which gives us that
$$
X^{x_1}(t) - X^{x_2}(t) = x_1-x_2 + \int_0^t \left(b(X^{x_1}(s-)) - b(X^{x_2}(s-))\right)\,ds,\quad\forall\, t>0,
$$
and, using the fact that the drift coefficients $b(x)$ are Lipschitz continuous functions, we obtain
\begin{align*}
|X^{x_1}(t) - X^{x_2}(t)| \leq |x_1-x_2| + [b]_{C^{0,1}(\RR^n)} \int_0^t |X^{x_1}(s-) - X^{x_2}(s-)|\,ds,\quad\forall\, t>0.
\end{align*}
Gronwall's inequality now gives us that
\begin{align*}
|X^{x_1}(t) - X^{x_2}(t)| \leq |x_1-x_2| e^{\beta t},\quad\forall\, t>0,
\end{align*}
where we denote for brevity $\beta:=[b]_{C^{0,1}(\RR^n)}$. The preceding inequality together with \eqref{eq:Ineq_tau_stat} imply
\begin{equation}
\label{eq:Diff_x_better_stat}
|v(x_1;\tau)-v(x_2;\tau)| \leq C\left(|x_1-x_2|(e^{(\beta - c_0) T} + 1) +e^{-c_0T}\right),\quad\forall\,\tau\in\cT.
\end{equation} 
Letting $\gamma:=c_0/\beta$ and choosing $T>0$ large enough such that
\begin{equation}
\label{eq:Choice_T}
e^{-c_0T} = |x_1-x_2|^{\gamma},
\end{equation}
we have that
\begin{equation}
\label{eq:Ineq_diff_x}
|x_1-x_2|e^{(\beta - c_0) T} = |x_1-x_2|^{1+\gamma-\gamma\beta/c_0} = |x_1-x_2|^{\gamma}.
\end{equation}
Letting now $\alpha:=1\wedge\gamma$, it follows from estimates \eqref{eq:Diff_x_better_stat}, \eqref{eq:Choice_T}, and \eqref{eq:Ineq_diff_x} that 
\begin{equation}
\label{eq:Ineq_difference_stat_est}
|v(x_1;\tau)-v(x_2;\tau)| \leq C|x_1-x_2|^{\alpha},\quad\forall\, x_1,x_2\in\RR^n,\quad\forall\,\tau\in\cT.
\end{equation}
Thus, using identity \eqref{eq:Ineq_difference_stat} we obtain that the value function $v$ belongs to $C^{\alpha}(\RR^n)$. When inequality \eqref{eq:Cond_for_Lipschitz_stat} holds, the fact that the value function $v$ belongs to $C^{0,1}(\RR^n)$ is an immediate consequence of the fact that $\gamma\geq 1$, and so $\alpha = 1$. This completes the proof.
\end{proof}

For all $r>0$ and $x\in\RR^n$, we let
\begin{equation}
\label{eq:tau_r}
\tau_r := \inf\{t \geq 0: X^x(t)\notin B_r(x)\},
\end{equation}
where $\{X^x(t)\}_{t\geq 0}$ is the unique
solution to equation \eqref{eq:Process} with initial condition
$X^x(0)=x$ and $B_r(x)$ is the Euclidean ball of radius $r$ centered at $x$. We prove Theorem~\ref{thm:Existence_stat} with the aid of the following Dynamic Programming Principle.

\begin{lem}[Dynamic Programming Principle]
\label{lem:DPP_stat}
Suppose that the hypotheses of Proposition~\ref{prop:Regularity_stat} hold. Then the value function $v(x)$ defined in \eqref{eq:Value_function_stat} satisfies:
\begin{equation}
\label{eq:DPP_stat}
v(x) = \sup\{v(x;r,\tau):\,\tau \leq \tau_r\},\quad\forall\, r>0, 
\end{equation}
where we define
\begin{equation}
\label{eq:DPP_stat_aux}
\begin{aligned}
v(x;r,\tau) &:=\EE\left[
e^{-\int_0^{\tau}c(X^x(s))\,ds} \left(\varphi(X^x(\tau))\mathbf{1}_{\{\tau<\tau_r\}}
+ v(X^x(\tau))\mathbf{1}_{\{\tau = \tau_r\}}\right)\right]\\
&\quad + \EE\left[\int_0^{\tau\wedge\tau_r} e^{-\int_0^t c(X^x(s))\, ds} f(X^x(t))\, dt
\right].
\end{aligned}
\end{equation}
\end{lem}

\begin{proof}
We denote by $w(x)$ the right hand-side of identity \eqref{eq:DPP_stat}, and we divide the proof into two steps.

\setcounter{step}{0}
\begin{step}[Proof of inequality $v(x)\leq w(x)$]
\label{step:Ineq_v_less_w} 
Let $\tau\in\cT$. From definition \eqref{eq:Value_function_stat_aux} of the function $v(x;\tau)$, conditioning on the $\sigma$-algebra $\cF_{\tau_r}$, we have that
\begin{equation}
\label{eq:Ineq_v_tau}
\begin{aligned}
&v(x;\tau) 
= 
\EE \left[e^{-\int_0^{\tau} c(X^x(s))\, ds} \varphi(X^x(\tau))\mathbf{1}_{\{\tau<\tau_r\}} 
+
\int_0^{\tau\wedge\tau_r} e^{-\int_0^ t c(X^x(s))\, ds} f(X^x(t))\, dt\right]\\
&\qquad+
\EE \left[\mathbf{1}_{\{\tau\geq\tau_r\}}e^{-\int_0^{\tau_r} c(X^x(s))\, ds}\right.\\
&\qquad\quad\times\left.
\EE \left[ e^{-\int_{\tau_r}^{\tau} c(X^x(s))\, ds} \varphi(X^x(\tau))
+ \int_{\tau_r}^{\tau} e^{-\int_{\tau_r}^t c(X^x(s))\, ds} f(X^x(t))\, dt \Big{|} \cF_{\tau_r}\right] \right].
\end{aligned}
\end{equation}
Let $\theta :=(\tau-\tau_r)\vee 0$. We next prove that the last term in the preceding identity can be written in the form 
\begin{equation}
\label{eq:Last_term}
\begin{aligned}
v(X^x(\tau_r);\theta) 
&= 
\EE \left[ e^{-\int_{\tau_r}^{\tau} c(X^x(s))\, ds} \varphi(X^x(\tau))
+ \int_{\tau_r}^{\tau} e^{-\int_{\tau_r}^t c(X^x(s))\, ds} f(X^x(t))\, dt \Big{|} \cF_{\tau_r}\right],\\
&\qquad\PP\hbox{-a.s.\ on }\{\tau \geq \tau_r\}. 
\end{aligned}
\end{equation}
For this purpose we need to show that there is a regular conditional probability distribution of $\PP$ given $\cF_{\tau_r}$, and that $\theta$ is a $\{\cF_{\tau_r+t}\}_{t\geq 0}$-stopping time.

To construct regular conditional probability distributions, we can assume without loss of generality that the sample space $\Omega$ of the Poisson random measure $N(dt, dy)$ with L\'evy measure $\nu(dy)$ is the space of functions $\omega:[0,\infty)\rightarrow\RR^n$ that are right-continuous and have left-limits (RCLL). We endow the space $\Omega$ with the Borel $\sigma$-algebra generated by the Skorohod topology, \cite[Chapter~3.5]{Ethier_Kurtz}, which we denote by $\cG$. For all $t\geq 0$, we let $\cG_t$ denote the $\sigma$-algebra generated by the Skorohod topology on $\Omega_t:=\{\omega:\,\omega:[0,t]\rightarrow\RR^n \hbox{ RCLL}\}$. Letting $\cG'$($\cG'_t$) be the $\sigma$-algebra generated by $\cG$ ($\cG_t$) and the $\PP$-null sets of $\cG$, and setting $\cF:=\cG'$ and $\cF_t = \cap_{s>t}\cG_s'$, we obtain that the filtered probability space $(\Omega, \cF, \{\cF_t\}_{t\geq 0}, \PP)$ satisfies the usual hypotheses in \cite[\S~I.1]{Protter}. Moreover, applying \cite[Theorem~3.5.6]{Ethier_Kurtz} we obtain that $\Omega$ endowed with the Skorohod topology is a complete separable metric space, and so \cite[Theorems~1.1.6 and 1.1.8]{Stroock_Varadhan} give us that there is a regular conditional probability distribution of $\PP$ given $\cF_{\tau_r}$, which we denote by $\PP_{\tau_r}$. To conclude the proof of Step~\ref{step:Ineq_v_less_w}, we next prove

\begin{locclaim}
\label{claim:theta}
The random time $\theta:=(\tau-\tau_r)\vee 0$ is a $\{\cF_{\tau_r+t}\}_{t\geq 0}$-stopping time.
\end{locclaim}

\begin{proof}[Proof of Claim~\ref{claim:theta}] 
We will show that for all $t\geq 0$ we have that $\{\theta \leq t\} \in \cF_{\tau_r+t}$, which is equivalent to proving that for all $t<s$ we have that the event $S:=\{\tau\leq\tau_r + t\} \cap \{\tau_r+t \leq s\}$ belongs to $\cF_s$. Because the filtration $\{\cF_t\}_{t\geq 0}$ is right-continuous, the preceding inclusion can be replaced with the condition 
\begin{equation}
\label{eq:Inclusion_s_t}
S:=\{\tau < \tau_r + t\} \cap \{\tau_r+t < s\} \in \cF_s,\quad\forall\, 0\leq t<s.
\end{equation}
The event $S$ can be written in the form 
$$
S = \cup_{n\in\NN}\cup_{q\in [0,s-t)\cap\QQ}\left(\{q-1/n<\tau_r<q\} \cap \{\tau<q-1/n+t\}\right),
$$
which clearly implies that they belong to the $\sigma$-algebra $\cF_s$. Thus property \eqref{eq:Inclusion_s_t} is satisfied, which concludes the proof of Claim~\ref{claim:theta}.
\end{proof}

Applying Claim~\ref{claim:theta} and the strong Markov property of the process $\{X^x(t)\}_{t\geq 0}$, we obtain that
\begin{align*}
&\EE \left[ e^{-\int_{\tau_r}^{\tau} c(X^x(s))\, ds} \varphi(X^x(\tau))
+ \int_{\tau_r}^{\tau} e^{-\int_{\tau_r}^t c(X^x(s))\, ds} f(X^x(t))\, dt \Big{|} \cF_{\tau_r}\right]\\
&\qquad =
\EE_{\PP_{\tau_r}} 
\left[ e^{-\int_0^{\theta} c(X^x(s))\, ds} \varphi(X^x(\theta))+ \int_0^{\theta} e^{-\int_0^t c(X^x(s))\, ds} f(X^x(t))\, dt\right],
\quad \PP\hbox{-a.s.\ on }\{\tau \geq \tau_r\},
\end{align*}
which together with definition \eqref{eq:Value_function_stat_aux} give us that identity \eqref{eq:Last_term} holds.
Identities \eqref{eq:Ineq_v_tau} and \eqref{eq:Last_term} imply that
\begin{align*}
v(x;\tau) 
&\leq 
\EE \left[e^{-\int_0^{\tau} c(X^x(s))\, ds} \varphi(X^x(\tau))\mathbf{1}_{\{\tau<\tau_r\}}
+ \int_0^{\tau\wedge\tau_r} e^{-\int_0^ t c(X^x(s))\, ds} f(X^x(t))\, dt\right]\\
&\quad+
\EE \left[\mathbf{1}_{\{\tau\geq\tau_r\}}e^{-\int_0^{\tau_r} c(X^x(s))\, ds} v(X^x(\tau_r)) \right]
\quad\hbox{(because $v(X^x(\tau_r;\theta)) \leq v(X^x(\tau_r))$ by \eqref{eq:Value_function_stat})}\\
&= v(x;r,\tau\wedge\tau_r)\quad\hbox{ (by \eqref{eq:DPP_stat_aux}).}
\end{align*}
Inequality $v(x)\leq w(x)$ now follows from the preceding expression by taking the supremum over all stopping times $\tau\in\cT$. 
\end{step}

\begin{step}[Proof of inequality $w(x)\leq v(x)$]
\label{step:Ineq_w_less_v} 
We fix $\eps\in (0,1)$, and we choose a stopping time $\tau=\tau(\eps)\leq\tau_r$ with the property that 
\begin{equation}
\label{eq:w_almost_less_v}
w(x)<v(x;r,\tau) + \eps.
\end{equation}
We choose a countable family of Borel measurable sets, $\{A_k\}_{k\in\NN}$, that partitions $\RR^n$, and a sequence of points, $\{x_k\}_{k\in\NN}$, such that $x_k\in A_k \subseteq B_{\eps}(x_k)$. For all $k\in\NN$, we choose a stopping time, $\theta_k\in\cT$, such that 
\begin{equation}
\label{eq:v_x_k}
v(x_k) < v(x_k;\theta_k) +\eps.
\end{equation}
Let $\alpha\in (0,1)$ be the H\"older exponent appearing in the statement of Proposition~\ref{prop:Regularity_stat}. For all $\omega\in\{\tau = \tau_r \hbox{ and } X^x(\tau_r) \in A_k\}$, we have that
\begin{align*}
v(X^x(\tau_r)) &\leq v(x_k) + C \eps^{\alpha}
\quad\hbox{ (by \eqref{eq:Ineq_difference_stat_est} and the fact that $A_k\subseteq B_{\eps}(x_k)$)}\\
&\leq v(x_k;\theta_k) + \eps + C\eps^{\alpha}
\quad\hbox{ (by our choice of $\theta_k$ in \eqref{eq:v_x_k})}\\
&\leq v(X^x(\tau_r);\theta_k) + (v(x_k;\theta_k) - v(X^x(\tau_r);\theta_k)) + \eps + C\eps^{\alpha}\\
&\leq v(X^x(\tau_r);\theta_k) + \eps + 2C\eps^{\alpha}
\quad\hbox{ (by \eqref{eq:Ineq_difference_stat_est} and the fact that $A_k\subseteq B_{\eps}(x_k)$).}
\end{align*}
The preceding construction together with definition \eqref{eq:DPP_stat_aux} of $v(x;r,\tau)$ and inequality \eqref{eq:w_almost_less_v} yield
\begin{equation}
\label{eq:w_v_bar_tau}
\begin{aligned}
w(x) &< 
\EE\left[e^{-\int_0^{\tau}c(X^x(s))\,ds} \varphi(X^x(\tau))\mathbf{1}_{\{\tau<\tau_r\}}
+ \int_0^{\tau\wedge\tau_r} e^{-\int_0^t c(X^x(s))\, ds} f(X^x(t))\, dt \right] \\
&\quad
+\EE\left[ e^{-\int_0^{\tau}c(X^x(s))\,ds} 
\sum_{k\in\NN} v(X^x(\tau_r);\theta_k) \mathbf{1}_{\{\tau = \tau_r\}} \mathbf{1}_{\{X^x(\tau_r) \in A_k\}}\right] + C\eps^{\alpha}.
\end{aligned}
\end{equation}
We notice that, by letting $S_{\tau_r}(\omega)(t):=\omega(\tau_r+t)$ denote the shift operator, we have that

\begin{locclaim}
\label{claim:bar_tau}
The random time defined by
\begin{equation}
\label{eq:bar_tau}
\bar\tau:= \tau_r+\sum_{k\in\NN} \theta_k(S_{\tau_r}) \mathbf{1}_{\{\tau=\tau_r\}} \mathbf{1}_{\{X^x(\tau_r) \in A_k\}}
\end{equation}
is a $\{\cF_{t}\}_{t\geq 0}$-stopping time. 
\end{locclaim}

\begin{proof}[Proof of Claim~\ref{claim:bar_tau}]
Because $\tau_r$ is a $\{\cF_{t}\}_{t\geq 0}$-stopping time, the conclusion of Claim~\ref{claim:bar_tau} follows once we prove that the random time $\bar\theta := \bar\tau-\tau_r$ is a $\{\cF_{\tau_r+t}\}_{t\geq 0}$-stopping time. For this purpose, it is sufficient to prove that the set $E:=\{\bar\theta\leq t\}$ is $\cF_{\tau_r+t}$-measurable, for all $t\geq 0$. From the definitions of $\bar\theta$ and $\bar\tau$, using the fact that the family of sets $\{A_k\}_{k\in\NN}$ partitions $\RR^n$, we can write $E$ as a disjoint union of sets $\{E_k\}_{k\in\NN}$, where
$$
E_k:=\{X^x(\tau_r)\in A_k\} \cap \{\tau=\tau_r\}\cap \{\theta_k(S_{\tau_r}) \leq t\},\quad\forall\, k\in\NN.
$$
The proof is concluded once we prove that the set $E_k$ is $\cF_{\tau_r+t}$-measurable, for all $k\in\NN$. Using the fact that the shift operator $S_{\tau_r}:(\Omega,\cF_{\tau_r+t})\rightarrow (\Omega,\cF)$ is a measurable function, and $\theta_k$ is a $\{\cF_t\}_{t\geq 0}$-stopping time, we obtain that the event $\{\theta_k(S_{\tau_r})\leq t\}$ is $\cF_{\tau_r+t}$-measurable. Combining this property with the fact that $\{X^x(S_{\tau_r})\in A_k\}$ and $\{\tau=\tau_r\}$ are $\cF_{\tau_r}$-measurable sets, we deduce that $E_k$ is $\cF_{\tau_r+t}$-measurable. Hence, the set $E$ is $\cF_{\tau_r+t}$-measurable, for all $t\geq 0$, which completes the proof of Claim~\ref{claim:bar_tau}.
\end{proof}

Using definition \eqref{eq:Value_function_stat_aux}, we have that
\begin{align*}
&\sum_{k\in\NN} v(X^x(\tau_r);\theta_k) \mathbf{1}_{\{\tau = \tau_r\}} \mathbf{1}_{\{X^x(\tau_r) \in A_k\}}\\ 
&\qquad = \mathbf{1}_{\{\tau = \tau_r\}} 
\EE_{\PP_{\tau_r}} 
 \left[e^{-\int_{\tau_r}^{\bar\tau} c(X^x(s))\, ds} \varphi(X^x(\bar\tau))
+ \int_{\tau_r}^{\bar\tau} e^{-\int_0^ t c(X^x(s))\, ds} f(X^x(t))\, dt\right].
\end{align*}
Recalling that $\PP_{\tau_r}$ is a regular condition probability distribution given $\cF_{\tau_r}$, it follows from the preceding identity and \eqref{eq:w_v_bar_tau} that
\begin{align*}
w(x) &< 
\EE\left[e^{-\int_0^{\tau}c(X^x(s))\,ds} \varphi(X^x(\tau))\mathbf{1}_{\{\tau<\tau_r\}}
+ \int_0^{\tau\wedge\tau_r} e^{-\int_0^t c(X^x(s))\, ds} f(X^x(t))\, dt \right] \\
&\quad
+\EE\left[ e^{-\int_0^{\tau}c(X^x(s-))\,ds} \mathbf{1}_{\{\tau = \tau_r\}} \right.\\
&\qquad\quad\times\left.
\EE\left[e^{-\int_{\tau_r}^{\bar\tau} c(X^x(s))\, ds} \varphi(X^x(\bar\tau))
+ \int_{\tau_r}^{\bar\tau} e^{-\int_{\tau_r}^ t c(X^x(s))\, ds} f(X^x(t))\, dt \Big{|}\cF_{\tau_r}\right]
\right]\\
&\quad + C\eps^{\alpha}.
\end{align*}
Claim~\ref{claim:bar_tau} yields that $\bar\tau$ belongs to $\cT$, and using definition \eqref{eq:Value_function_stat_aux} we can write the sum of the first two terms on the right-hand side of the preceding inequality as $v(x;\bar\tau)$. Hence, we deduce that $w(x) < v(x;\bar\tau)+c\eps^{\alpha}$, from which it follows that $w(x) < v(x)+C\eps^{\alpha}$ by identity \eqref{eq:Value_function_stat}. We obtain the desired inequality $w(x)\leq v(x)$ by letting $\eps$ tend to zero. 
\end{step}

Steps~\ref{step:Ineq_v_less_w}  and \ref{step:Ineq_w_less_v} combined complete the proof.
\end{proof}

We next apply Lemma~\ref{lem:DPP_stat} to give the

\begin{proof}[Proof of Theorem~\ref{thm:Existence_stat}]
We divide the proof into two steps.

\setcounter{step}{0}
\begin{step}[Verification that $v$ is a subsolution]
\label{step:Subsol_stat}
To prove that $v$ is a subsolution to equation \eqref{eq:Stationary_obs_problem}, let $u\in C^2(\RR^n)$ be such that $u(x_0)=v(x_0)$ and $u(x)\geq v(x)$, for all $x\in\RR^n$. We need to show that
$$
-Lu(x_0)+c(x_0) u(x_0) \leq f(x_0)\quad\hbox{ or }\quad u(x_0) \leq \varphi(x_0).
$$
We assume without loss of generality that $u(x_0) > \varphi(x_0)$, otherwise the conclusion that $v$ is a subsolution is obtained. For all stopping times $\tau\in\cT$, It\^o's rule applied to $u$ and the unique strong solution, $\{X^{x_0}(t)\}_{t\geq 0}$, to equation \eqref{eq:Process} with initial condition $X^{x_0}(0)=x_0$, gives us 
\begin{align*}
&e^{-\int_0^{\tau} c(X^{x_0}(s))\,ds} u(X^{x_0}(\tau))\\
&\qquad\quad = u(x_0)+\int_0^{\tau} e^{-\int_0^t c(X^{x_0}(s))\,ds} (L-c(X^{x_0}(s))) u(X^{x_0}(t))\, dt + M(\tau),
\end{align*}
where we define the process
$$
M(t):=\int_0^t\int_{\RR^n\setminus\{O\}}e^{-\int_0^s c(X^{x_0}(r))\,dr}\left(u(X^{x_0}(s-)+F(X^{x_0}(s-),y)) - u(X^{x_0}(s-))\right)\, \widetilde N(ds,dy),
$$
for all $t\geq 0$. It follows by Proposition~\ref{prop:Regularity_stat}, condition \eqref{eq:Jump_mart_stat}, and \cite[Theorem~4.2.3]{Applebaum} that the process $\{M(t)\}_{t\geq 0}$ is a martingale, which gives us that
\begin{align*}
u(x_0) &= \EE\left[e^{-\int_0^{\tau} c(X^{x_0}(s))\,ds} u(X^{x_0}(\tau)) \right] \\
&\quad+ 
\EE\left[\int_0^{\tau} e^{-\int_0^t c(X^{x_0}(s))\,ds} (-L+c(X^{x_0}(t))) u(X^{x_0}(t))\, dt\right].
\end{align*}
Assume by contradiction that there are positive constants, $\eps$ and $r$, such that $(-L+c) u(x)\geq f(x)+\eps$, for all $x\in B_r(x_0)$. Using the fact that $v(x_0)=u(x_0)$, and replacing the stopping times $\tau$ in the preceding identity for $u(x_0)$ by $\tau\wedge\tau_r$, it follows that
\begin{align*}
v(x_0) &\geq \EE\left[e^{-\int_0^{\tau\wedge\tau_r} c(X^{x_0}(s))\,ds} u(X^{x_0}(\tau\wedge\tau_r)) \right]\notag\\
&\quad + \EE\left[\int_0^{\tau\wedge\tau_r} e^{-\int_0^t c(X^{x_0}(s))\,ds}\left(f(X^{x_0}(t))+\eps\right)\, dt\right].
\end{align*}
Letting $M:=\|c\|_{C(\RR^n)}$ and using the fact that
\begin{align*}
\EE\left[\int_0^{\tau\wedge\tau_r} e^{-\int_0^t c(X^{x_0}(s))\,ds} \, dt\right]
&\geq \EE\left[\int_0^{\tau\wedge\tau_r} e^{-Mt} \, dt\right]\\
& = \EE\left[\frac{1}{M}\left(1 - e^{-M\tau\wedge\tau_r}\right) \right]\\
&\geq \EE\left[\frac{1}{2M} \mathbf{1}_{\{e^{-M\tau\wedge\tau_r}\leq 1/2\}} 
+ \frac{1}{2}\tau\wedge\tau_r \mathbf{1}_{\{e^{-M\tau\wedge\tau_r} > 1/2\}}\right]\\
&\geq \frac{1}{2}\EE\left[ \tau\wedge\tau_r \wedge \frac{1}{M} \right],
\end{align*}
it follows that
\begin{equation}
\label{eq:Ineq_stat_v_eps}
\begin{aligned}
v(x_0) 
&\geq \EE\left[e^{-\int_0^{\tau\wedge\tau_r} c(X^{x_0}(s))\,ds} u(X^{x_0}(\tau\wedge\tau_r)) 
+\int_0^{\tau\wedge\tau_r} e^{-\int_0^t c(X^{x_0}(s))\,ds} f(X^{x_0}(t))\, dt\right]\\ 
&\quad +\frac{\eps}{2}\EE\left[ \tau\wedge\tau_r \wedge \frac{1}{M} \right].
\end{aligned}
\end{equation}
Recall that $u(x_0) > \varphi(x_0)$, $u(x_0)=v(x_0)$, and $u\geq v$ on $\RR^n$, and so we can assume without loss of generality that $r$ is chosen small enough so that $u(x)\geq v(x)>\varphi(x)$, for all $x\in B_r(x_0)$. This implies that
\begin{align*}
&\sup_{\tau\in\cT} \EE\left[e^{-\int_0^{\tau\wedge\tau_r} c(X^{x_0}(s))\,ds} u(X^{x_0}(\tau\wedge \tau_r)) 
+\int_0^{\tau\wedge\tau_r} e^{-\int_0^t c(X^{x_0}(s))\,ds} f(X^{x_0}(t))\, dt\right] \\
&\quad\geq\sup_{\tau\in\cT} \EE\left[e^{-\int_0^{\tau\wedge\tau_r} c(X^{x_0}(s))\,ds} 
\left(\varphi(X^{x_0}(\tau\wedge \tau_r))\mathbf{1}_{\{\tau<\tau_r\}} +  v(X^{x_0}(\tau\wedge \tau_r))\mathbf{1}_{\{\tau\geq\tau_r\}}\right)\right]\\
&\qquad+\EE\left[\int_0^{\tau\wedge\tau_r} e^{-\int_0^t c(X^{x_0}(s))\,ds} f(X^{x_0}(t))\, dt\right] \\
&\quad= v(x_0),
\end{align*}
where in the last line we applied Lemma~\ref{lem:DPP_stat}. Thus, for all $\delta>0$, we can choose a stopping time $\tau^{\delta}\in\cT$ with the property that
\begin{equation}
\label{eq:Subsol_stat_contrad_2}
\begin{aligned}
&\EE\left[e^{-\int_0^{\tau^{\delta}\wedge\tau_r} c(X^{x_0}(s))\,ds} \left(\varphi(X^{x_0}(\tau^{\delta}\wedge\tau_r)) 
\mathbf{1}_{\{\tau^{\delta}<\tau_r\}} +  v(X^{x_0}(\tau^{\delta}\wedge \tau_r))\mathbf{1}_{\{\tau^{\delta}\geq\tau_r\}}\right)\right]\\
&\quad+\EE\left[\int_0^{\tau^{\delta}\wedge\tau_r} e^{-\int_0^t c(X^{x_0}(s))\,ds} f(X^{x_0}(t))\, dt\right] > v(x_0) - \delta.
\end{aligned}
\end{equation}
Letting $\tau=\tau^{\delta}$ in inequality \eqref{eq:Ineq_stat_v_eps} and using \eqref{eq:Subsol_stat_contrad_2}, it follows that, for all $\delta>0$, we have the inequality
\begin{equation}
\label{eq:Subsol_stat_contrad}
v(x_0) \geq v(x_0)-\delta + \frac{\eps}{2}\EE\left[\tau\wedge\tau_r \wedge \frac{1}{M} \right].
\end{equation}
We next claim that
\begin{equation}
\label{eq:Subsol_stat_contrad_1}
\liminf_{\delta\rightarrow 0} \EE\left[\tau^{\delta}\wedge\tau_r\wedge \frac{1}{M}\right] > 0.
\end{equation}
When property \eqref{eq:Subsol_stat_contrad_1} does not hold, there is
a sequence $\{\delta_k\}_{k\in\NN}$ such that
$\{\EE\left[\tau^{\delta_k}\wedge\tau_r\wedge \frac{1}{M}\right]\}_{k\in\NN}$ converges to zero, and $\{\tau^{\delta_k}\wedge\tau_r\}_{k\in\NN}$ also converges to zero $\PP$-a.s.. This property and the Dominated Convergence Theorem us by inequality \eqref{eq:Subsol_stat_contrad_2} that 
$$
\varphi(x_0) \PP_{x_0}(\tau_r>0) + v(x_0)\PP_{x_0}(\tau_r=0) \geq v(x_0).
$$
This implies that $\varphi(x_0) \PP_{x_0}(\tau_r>0) \geq v(x_0) \PP_{x_0}(\tau_r>0)$, and using the fact that the event $\tau_r>0$ has positive probability, we obtain the inequality $\varphi(x_0) \geq v(x_0)$, which contradicts our assumption that $u(x_0)=v(x_0)>\varphi(x_0)$. Thus, property \eqref{eq:Subsol_stat_contrad_1} holds. Combining now \eqref{eq:Subsol_stat_contrad_1} with inequality \eqref{eq:Subsol_stat_contrad}, where we let $\delta$ tend to zero, we obtain the contradiction $v(x_0)>v(x_0)$. This implies that
$-Lu(x_0)+c(x_0) u(x_0) \leq f(x_0)$, which guarantees that
$$
\min\{-Lu(x_0)+c(x_0) u(x_0) - f(x_0), u(x_0)-\varphi(x_0)\}\leq 0,
$$ 
and so, indeed $v$ defined in \eqref{eq:Value_function_stat} is a viscosity subsolution to the obstacle problem \eqref{eq:Stationary_obs_problem}. This concludes the proof of Step~\ref{step:Subsol_stat}.
\end{step}

\begin{step}[Verification that $v$ is a supersolution]
\label{step:Supersol_stat}
Let $u\in C^2(\RR^n)$ be such that $u(x_0)=v(x_0)$ and $u(x)\leq v(x)$, for all $x\in\RR^n$. Our goal is to prove that
$$
\min\{-Lu(x_0) + c(x_0)u(x_0)-f(x_0), u(x_0)-\varphi(x_0)\} \geq 0,
$$
and so conclude that the value function $v$ is a supersolution to the obstacle problem \eqref{eq:Stationary_obs_problem}. Definition~\eqref{eq:Value_function_stat} implies that $v(x)\geq\varphi(x)$, for all $x\in\RR^n$, and in particular that $u(x_0)\geq\varphi(x_0)$. Thus, we need to show that 
\begin{equation}
\label{eq:Ineq_supersol_stat}
-Lu(x_0) + c(x_0)u(x_0)\geq f(x_0).
\end{equation}
Assuming by contradiction that the preceding inequality does not hold, there are positive constants, $\eps$ and $r$, such that
\begin{equation}
\label{eq:Ineq_supersol_stat_contrad}
-Lu(x) + c(x)u(x) \leq f(x)-\eps,\quad\,\forall\, x\in B_r(x_0).
\end{equation}
Using the fact that $u(x_0)=v(x_0)$, and applying It\^o's rule to $u$ and the unique strong solution, $\{X^{x_0}(t)\}_{t\geq 0}$, to equation \eqref{eq:Process} with initial condition $X^{x_0}(0)=x_0$, we obtain similarly to Step~\ref{step:Subsol_stat} that
\begin{equation*}
\begin{aligned}
v(x_0) &= \EE\left[e^{-\int_0^{\tau_r} c(X^{x_0}(s))\,ds} u(X^{x_0}(\tau_r))  + 
\int_0^{\tau_r} e^{-\int_0^t c(X^{x_0}(s))\,ds} (-L+c(X^{x_0}(t-))) u(X^{x_0}(t))\, dt\right]\notag\\
& \leq \EE\left[e^{-\int_0^{\tau_r} c(X^{x_0}(s))\,ds} v(X^{x_0}(\tau_r))  + 
\int_0^{\tau_r} e^{-\int_0^t c(X^{x_0}(s))\,ds} f(X^{x_0}(t))\, dt\right]\notag\\
&
\quad - \eps \EE\left[\int_0^{\tau_r} e^{-\int_0^t c(X^{x_0}(s))\,ds}\, dt\right],
\end{aligned}
\end{equation*}
where in the second inequality we used the fact that $u(x)\leq v(x)$, for all $x\in\RR^n$, together with assumption \eqref{eq:Ineq_supersol_stat_contrad}. Notice that the first term in the second inequality above is $v(x_0;r,\tau_r)$ by identity \eqref{eq:DPP_stat_aux}. Applying Lemma~\ref{lem:DPP_stat} we have that $v(x_0;r,\tau_r)\leq v(x_0)$, which gives us that
$$
v(x_0) \leq v(x_0) - \eps \EE_{x_0}\left[\int_0^{\tau_r} e^{-\int_0^t c(X(s))\,ds}\, dt\right].
$$	
Because $\PP_{x_0}(\tau_r>0)$ is positive, we obtain the contradiction $v(x_0)<v(x_0)$. This implies that inequality \eqref{eq:Ineq_supersol_stat} holds, which shows that $v$ is a supersolution to equation \eqref{eq:Stationary_obs_problem}.
\end{step}

Steps~\ref{step:Subsol_stat} and \ref{step:Supersol_stat} complete the proof that the value function defined in \eqref{eq:Value_function_stat} is a viscosity solution to the obstacle problem \eqref{eq:Stationary_obs_problem}.
\end{proof}

We prove Theorem~\ref{thm:Uniqueness_stat} with the aid of the following comparison principle:

\begin{thm}[Comparison principle]
\label{thm:Comparison_princ_stat}
Suppose that Assumption~\ref{assump:Coeff} holds, $c \in C(\RR^n)$ satisifes condition \eqref{eq:c_lower_bound}, $f\in C(\RR^n)$, and condition \eqref{eq:Interior_point} holds. If $u$ and $v$ are a viscosity subsolution and supersolution to the obstacle problem \eqref{eq:Stationary_obs_problem}, respectively, then $u\leq v$.
\end{thm}

\begin{proof}
As usual in comparison arguments for viscosity solutions on unbounded domains, \cite[Proof of Theorem~5.1]{Crandall_Ishii_Lions_1992}, we let $\alpha$ and $\eps$ be positive constants and we define:
\begin{equation}
\label{eq:M_a_e_stat}
M_{\alpha,\eps} := \sup_{x,y\in\RR^n} \left\{u(x)-v(y) - \frac{\alpha}{2}|x-y|^2-\eps(|x|^2+|y|^2)\right\}.
\end{equation}
Letting $x_{\alpha,\eps}$ and $y_{\alpha,\eps}$ be points where the supremum of $M_{\alpha,\eps}$ is attained, it follows by \cite[Lemma~3.1]{Crandall_Ishii_Lions_1992} that
\begin{equation}
\label{eq:Conv_diff_x_y_stat}
\alpha|x_{\alpha,\eps} - y_{\alpha,\eps}|^2\rightarrow 0,\quad\hbox{ as } \alpha\rightarrow\infty,
\end{equation}
and so, we can assume without loss of generality that
\begin{equation}
\label{eq:Conv_x_y_stat}
x_{\alpha,\eps}, y_{\alpha,\eps} \rightarrow x_{\eps},\quad\hbox{ as }\alpha\rightarrow\infty,
\end{equation}
and applying again \cite[Lemma~3.1]{Crandall_Ishii_Lions_1992}, we have that
\begin{align}
\label{eq:No_a_stat}
u(x_{\eps}) - v(x_{\eps}) -2\eps|x_{\eps}|^2 &= \sup_{x\in\RR^n} \{u(x)-v(x) - 2\eps|x|^2\}.
\end{align}
For all $\delta>0$, condition \eqref{eq:Interior_point} ensures that there is $r=r(\alpha,\delta,\eps)\in (0,\delta)$ such that 
$$
B_r\subseteq 
\{y\in\RR^n:|F(x_{\alpha,\eps}, y)| <\delta\,\hbox{ and }\,|F(y_{\alpha,\eps}, y)| <\delta\}.
$$  
We consider the the auxiliary function
\begin{equation}
\begin{aligned}
\hat u(x) &:= u(x_{\alpha,\eps}) + \frac{\alpha}{2}\left(|x-y_{\alpha,\eps}|^2-|x_{\alpha,\eps}-y_{\alpha,\eps}|^2\right)+\eps(|x|^2-|x_{\alpha,\eps}|^2),\notag\\
\label{eq:u_supersol_stat}
\bar u(x) &:= \left\{
\begin{array}{rl}
\hat u(x), &\quad \hbox{ if } x \in B_r(x_{\alpha,\eps}),\\
u(x), &\quad \hbox{ if } x \in (B_r(x_{\alpha,\eps}))^c,
\end{array} \right.
\end{aligned}
\end{equation}
for all $x\in\RR^n$. From definition \eqref{eq:M_a_e_stat} of $M_{\alpha,\eps}$ and from the choice of the points $x_{\alpha,\eps}$ and $y_{\alpha,\eps}$, we see that $u\leq \bar u$ on $\RR^n$ and $x_{\alpha,\eps}$ is a point at which $u-\bar u$ attains its maximum. Similarly, we define
\begin{equation}
\begin{aligned}
\hat v(y) &:= v(y_{\alpha,\eps}) + \frac{\alpha}{2}\left(|x_{\alpha,\eps}-y_{\alpha,\eps}|^2-|x_{\alpha,\eps}-y|^2\right)-\eps(|y|^2-|y_{\alpha,\eps}|^2),\notag\\
\label{eq:v_supersol_stat}
\bar v(y) &:= \left\{
\begin{array}{rl}
\hat v(y), &\quad \hbox{ if } y \in B_r(y_{\alpha,\eps}),\\
v(y), &\quad \hbox{ if } y \in (B_r(y_{\alpha,\eps}))^c,
\end{array} \right.
\end{aligned}
\end{equation}
for all $y\in\RR^n$, and we see that $v\geq \bar v$ on $\RR^n$ and $y_{\alpha,\eps}$ is a point at which $v-\bar v$ attains its minimum.
By construction, the auxiliary functions $\bar u$ and $\bar v$ are $C^2$ in a neighborhood of $x_{\alpha,\eps}$ and $y_{\alpha,\eps}$, respectively, and are continuous functions on $\RR^n\setminus\partial B_r(x_{\alpha,\eps})$ and $\RR^n\setminus\partial B_r(y_{\alpha,\eps})$, respectively. A mollification argument applied to $\bar u$ and $\bar v$ shows that even though $\bar u$ and $\bar v$ are not $C^2$ functions on $\RR^n$, we can still apply Definition~\ref{defn:Viscosity_sol_stat} to obtain that
\begin{equation}
\label{eq:Sup_super_aux_stat}
\begin{aligned}
\min\{-L\bar u(x_{\alpha,\eps}) + c(x_{\alpha,\eps})\bar u(x_{\alpha,\eps}) - f(x_{\alpha,\eps}), 
\bar u(x_{\alpha,\eps}) - \varphi(x_{\alpha,\eps})\} \leq\, 0,\\
\min\{-L\bar v(y_{\alpha,\eps}) + c(y_{\alpha,\eps}) \bar v(y_{\alpha,\eps}) - f(y_{\alpha,\eps}), 
\bar v(y_{\alpha,\eps}) - \varphi(y_{\alpha,\eps})\} \geq\, 0.
\end{aligned}
\end{equation}
For $\eps>0$ fixed, if there is a sequence $\{\alpha_k\}_{k\in\NN}$ converging to infinity such that $\bar u(x_{\alpha_k,\eps}) \leq \varphi(x_{\alpha_k,\eps})$, then it follows by \eqref{eq:Conv_x_y_stat} that $\bar u(x_{\eps}) \leq \varphi(x_{\eps})$. The second inequality in \eqref{eq:Sup_super_aux_stat} shows that $\bar v(y_{\alpha,\eps}) \geq \varphi(y_{\alpha,\eps})$, for all $\alpha>0$, and so we have that
\begin{equation}
\label{eq:Ineq_easy_case}
v(x_{\eps}) \geq \varphi(x_{\eps}) \geq u(x_{\eps}).
\end{equation}
If we cannot find a sequence $\{\alpha_k\}_{k\in\NN}$ satisfying the preceding property, then the inequalities in \eqref{eq:Sup_super_aux_stat} hold for the first terms on the left-hand side, which by subtraction imply that 
\begin{align}
c(x_{\alpha,\eps}) u(x_{\alpha,\eps}) - c(y_{\alpha,\eps}) v(y_{\alpha,\eps}) 
&= c(x_{\alpha,\eps}) \bar u(x_{\alpha,\eps}) - c(y_{\alpha,\eps}) \bar v(y_{\alpha,\eps})\notag\\
&\label{eq:Diff_op_sub_super_stat}
\leq L\bar u(x_{\alpha,\eps}) - L\bar v(y_{\alpha,\eps}) + f(x_{\alpha,\eps}) -  f(y_{\alpha,\eps}).
\end{align}
We write the last term in the preceding inequality as a sum $I_1+I_2+I_3$, where each term is defined by
\begin{align*}
I_1 &:=b(x_{\alpha,\eps})\dotprod\left(\alpha z_{\alpha,\eps}+2\eps x_{\alpha,\eps}\right)
-b(y_{\alpha,\eps})\dotprod\left(\alpha z_{\alpha,\eps}-2\eps y_{\alpha,\eps}\right)
+ f(x_{\alpha,\eps}) -  f(y_{\alpha,\eps}),\\
I_2 &:= \int_{B_r\setminus\{O\}} 
\left(\bar u(x_{\alpha,\eps}+F(x_{\alpha,\eps},y)) - \bar u(x_{\alpha,\eps})-\nabla\bar u(x_{\alpha,\eps})\dotprod F(x_{\alpha,\eps},y)\right) \,\nu(dy) \\
&\quad - \int_{B_r\setminus\{O\}}  \left(\bar v(y_{\alpha,\eps}+F(y_{\alpha,\eps},y)) - \bar v(y_{\alpha,\eps})-\nabla\bar v(y_{\alpha,\eps})\dotprod F(y_{\alpha,\eps},y)\right) \,\nu(dy),\\
I_3 &:=\int_{B_r^c} 
\left(\bar u(x_{\alpha,\eps}+F(x_{\alpha,\eps},y)) - \bar u(x_{\alpha,\eps})-\nabla\bar u(x_{\alpha,\eps})\dotprod F(x_{\alpha,\eps},y)\right) \,\nu(dy) \\
&\quad - \int_{B_r^c}  \left(\bar v(y_{\alpha,\eps}+F(y_{\alpha,\eps},y)) - \bar v(y_{\alpha,\eps})-\nabla\bar v(y_{\alpha,\eps})\dotprod F(y_{\alpha,\eps},y)\right) \,\nu(dy),
\end{align*}
where we denote $z_{\alpha,\eps}:=x_{\alpha,\eps} - y_{\alpha,\eps}$. 
Rearranging terms and using the definitions of the functions $\bar u$ and $\bar v$, we obtain by direct computations:
\begin{align*}
I_1 &:=\alpha\left(b(x_{\alpha,\eps})- b(y_{\alpha,\eps})\right)\dotprod  z_{\alpha,\eps}
+2\eps \left(b(x_{\alpha,\eps})\dotprod x_{\alpha,\eps} +b(y_{\alpha,\eps})\dotprod  y_{\alpha,\eps}\right)
+ f(x_{\alpha,\eps}) -  f(y_{\alpha,\eps}),\\
I_2 &:= \left(\frac{\alpha}{2}+\eps\right)\int_{B_r\setminus\{O\}} \left(|F(x_{\alpha,\eps},y)|^2 + |F(y_{\alpha,\eps},y)|^2\right) \,\nu(dy),\\
I_3 &:= \int_{B_r^c} \left(u(x_{\alpha,\eps}+F(x_{\alpha,\eps},y)) - u(x_{\alpha,\eps}) - v(y_{\alpha,\eps}+F(y_{\alpha,\eps},y)) + v(y_{\alpha,\eps})\right) \,\nu(dy)\\
&\quad
-\alpha \int_{B_r^c}z_{\alpha,\eps}\dotprod \left(F(x_{\alpha,\eps},y)- F(y_{\alpha,\eps},y)\right)\,d\nu(dy)\\
&\quad
-2\eps \int_{B_r^c} (x_{\alpha,\eps}\dotprod F(x_{\alpha,\eps},y) + y_{\alpha,\eps}\dotprod F(y_{\alpha,\eps},y))\dotprod  \,\nu(dy),
\end{align*}
Using definition \eqref{eq:M_a_e_stat} of $M_{\alpha,\eps}$, it follows that
\begin{align*}
&u(x_{\alpha,\eps}+F(x_{\alpha,\eps},y)) - u(x_{\alpha,\eps}) - v(y_{\alpha,\eps}+F(y_{\alpha,\eps},y)) + v(y_{\alpha,\eps})\\
&\qquad\leq 
\alpha z_{\alpha,\eps} \dotprod\left(F(x_{\alpha,\eps},y) - F(y_{\alpha,\eps},y)\right)\\
&\qquad\quad
+\frac{\alpha}{2}\left|F(x_{\alpha,\eps},y) - F(y_{\alpha,\eps},y)\right|^2
+2\eps x_{\alpha,\eps} \dotprod F(x_{\alpha,\eps},y) +2\eps y_{\alpha,\eps} \dotprod F(y_{\alpha,\eps},y)\\
&\qquad\quad
+\eps \left(|F(x_{\alpha,\eps},y)|^2 + |F(y_{\alpha,\eps},y)|^2\right),
\end{align*}
which implies that
\begin{align*}
I_3 &\leq \frac{\alpha}{2} \int_{B_r^c} |F(x_{\alpha,\eps},y) - F(y_{\alpha,\eps},y)|^2\,\nu(dy) 
+ \eps\int_{B_r^c} |\left(F(x_{\alpha,\eps},y)|^2 + |F(y_{\alpha,\eps},y)|^2\right)\,\nu(dy) \\
&
\leq K\left(\frac{\alpha}{2}|x_{\alpha,\eps}-y_{\alpha,\eps}|^2 + 2\eps\right),
\end{align*}
where in the last line we used conditions \eqref{eq:F_Lipschitz_squared} and \eqref{eq:F_Levy_cond}. Property \eqref{eq:F_Levy_cond} implies that $I_2$ converges to zero, as $r\downarrow 0$. Applying the preceding observations to the right-hand side of identity \eqref{eq:Diff_op_sub_super_stat} and letting $r$ tend to zero, we obtain
\begin{align*}
c(x_{\alpha,\eps}) u(x_{\alpha,\eps}) - c(y_{\alpha,\eps}) v(y_{\alpha,\eps}) 
&\leq \alpha\left(b(x_{\alpha,\eps})- b(y_{\alpha,\eps})\right)\dotprod z_{\alpha,\eps}
+2\eps \left(b(x_{\alpha,\eps})\dotprod x_{\alpha,\eps} +b(y_{\alpha,\eps})\dotprod y_{\alpha,\eps}\right)\\
&\quad+ f(x_{\alpha,\eps}) -  f(y_{\alpha,\eps}) + K\left(\frac{\alpha}{2}|z_{\alpha,\eps}|^2 + 2\eps\right).
\end{align*}
Because $b$ is a bounded, Lipschitz continuous function, we can find a positive constant, $C$, such that
\begin{align*}
c(x_{\alpha,\eps}) u(x_{\alpha,\eps}) - c(y_{\alpha,\eps}) v(y_{\alpha,\eps}) 
&\leq C \left(\alpha |z_{\alpha,\eps}|^2 + \eps +\eps(|x_{\alpha,\eps}|+|y_{\alpha,\eps}|)\right)\\
&\quad+|f(x_{\alpha,\eps}) -  f(y_{\alpha,\eps})|.
\end{align*}
Letting now $\alpha$ tend to infinity, and using properties \eqref{eq:Conv_diff_x_y_stat}, \eqref{eq:Conv_x_y_stat}, and the fact that $f\in C(\RR^n)$, it follows that 
\begin{equation}
\label{eq:Ineq_eps}
c(x_{\eps}) \left(u(x_{\eps}) - v(x_{\eps})\right) \leq C\left(\eps +\eps|x_{\eps}|\right),
\end{equation}
and using condition \eqref{eq:c_lower_bound} satisfied by the coefficient $c(x)$, we have
$$
u(x_{\eps}) - v(x_{\eps}) \leq C\left(\eps +\eps|x_{\eps}|\right).
$$
It is clear from identity \eqref{eq:No_a_stat} that $\eps|x_{\eps}|^2$ is bounded in $\eps$, and that
$$
\lim_{\eps\downarrow 0} (u(x_{\eps}) - v(x_{\eps})) = \sup_{x\in\RR^n} \{u(x)-v(x)\}.
$$
Letting now $\eps$ tend to zero in the preceding two properties, we obtain that $u(x)-v(x) \leq 0$, for all $x\in\RR^n$.
Together with inequality \eqref{eq:Ineq_easy_case}, this completes the proof. 
\end{proof}

\begin{rmk}
Examining the proof of Theorem~\ref{thm:Comparison_princ_stat}, we see from the inequality \eqref{eq:Ineq_eps} that 
$$
u(x_{\eps}) - v(x_{\eps}) \leq \frac{C}{c(x_{\eps})}\left(\eps +\eps|x_{\eps}|\right),
$$ 
where the right-hand side converges to zero when we assume that $c(x)$ is a positive function on $\RR^n$ and satisfies property 
\eqref{eq:c_lower_bound_infty}. Thus, as indicated in Remark \ref{eq:c_lower_bound_infty}, we again obtain that $u(x)-v(x) \leq 0$, for all $x\in\RR^n$, and the conclusion of Theorem~\ref{thm:Comparison_princ_stat} holds.
\end{rmk}

\begin{proof}[Proof of Theorem~\ref{thm:Uniqueness_stat}]
It is an obvious consequence of Theorem~\ref{thm:Comparison_princ_stat}.
\end{proof}

\section{Evolution obstacle problem}
\label{sec:Evolution}
In this section, we outline the proofs of Proposition~\ref{prop:Regularity_evol}, and Theorems~\ref{thm:Existence_evol} and \ref{thm:Uniqueness_evol}, and of a Dynamical Programming Principle in Lemma~\ref{lem:DPP_evol} and a comparison principle in Theorem~\ref{thm:Comparison_princ_evol}. Because the proofs are very similar to those for the stationary obstacle problem, we only point out the main changes that need to be done to the proofs in \S~\ref{sec:Stationary}. 

We begin with an auxiliary lemma which we use to prove Proposition~\ref{prop:Regularity_evol}:

\begin{lem}[Continuity properties of $\{X(t)\}_{t\geq 0}$]
\label{lem:Cont_X}
Suppose that Assumption~\ref{assump:Coeff} is satisfied. Then there is a positive constant, $C=C(\|b\|_{C^{0,1}(\RR^n)}, K)$ such that
\begin{align}
\label{eq:Cont_X_1}
\EE\left[\max_{s\in [0,t]} \left|X^{x_1}(s)-X^{x_2}(s)\right|^2\right] &\leq C|x_1-x_2|^2 e^{Ct},
\quad\forall\, x_1,x_2\in\RR^n,\, t\geq 0,\\
\label{eq:Cont_X_2}
\EE\left[\max_{r\in [s,t]} \left|X^x(r)-X^x(s)\right|^2\right] &\leq C|t-s|\vee|t-s|^2,\quad\forall\, x\in\RR^n,\, 0\leq s<t.
\end{align}
\end{lem}

\begin{proof}
To prove inequality \eqref{eq:Cont_X_1}, using the stochastic equation \eqref{eq:Process}, we have that
\begin{equation}
\label{eq:X_difference_x}
X^{x_1}(t)-X^{x_2}(t) = x_1-x_2 + \int_0^t(b(X^{x_1}(s)) - b(X^{x_2}(s)))\,ds + M(t),\quad\forall\,t \geq 0,
\end{equation}
where we denote by $\{M(t)\}_{t\geq 0}$ the square-integrable martingale:
$$
M(t) := \int_0^t\int_{\RR^n\setminus\{O\}}\left(F(X^{x_1}(s-),y) - F(X^{x_2}(s-),y)\right) \widetilde N(ds, dy).
$$
Applying Doob's martingale inequality \cite[Theorem~2.1.5]{Applebaum} and \cite[Lemma~4.2.2]{Applebaum} to $\{M(t)\}_{t\geq 0}$, and using property \eqref{eq:F_Lipschitz_squared}, it follows that
\begin{align*}
\EE\left[\sup_{s\in[0,t]} |M(s)|^2\right] &\leq 4 \EE\left[|M(t)|^2\right]\\
&\leq \EE\left[\int_0^t\int_{\RR^n\setminus\{O\}}\left|F(X^{x_1}(s-),y) - F(X^{x_2}(s-),y)\right|^2\, \nu(dy)\,ds\right]\\
&\leq K\EE\left[\int_0^t |X^{x_1}(s)-X^{x_2}(s)|^2\,ds\right].
\end{align*}
The preceding inequality, identity \eqref{eq:X_difference_x}, and the Lipschitz continuity of the drift coefficient $b(x)$, together with Gronwall's inequality, imply estimate \eqref{eq:Cont_X_1}. To prove inequality \eqref{eq:Cont_X_2}, we again use the stochastic equation \eqref{eq:Process} and we obtain that
\begin{equation}
\label{eq:X_difference_x_aux}
X^x(t)-X^x(s) = \int_s^t b(X^x(r))\,dr + N(t),\quad\forall\, 0\leq s<t,
\end{equation}
where we denote by $\{N(t)\}_{t\geq 0}$ the square-integrable martingale:
$$
N(t) := \int_s^t\int_{\RR^n\setminus\{O\}} F(X^x(r-),y) \widetilde N(dr, dy),\quad\forall\, t>s.
$$
Applying Doob's martingale inequality \cite[Theorem~2.1.5]{Applebaum} and \cite[Lemma~4.2.2]{Applebaum} to $\{N(t)\}_{t\geq 0}$, and using the boundedness of the coefficient $b(x)$, it follows from identity \eqref{eq:X_difference_x_aux} that
\begin{align*}
\EE\left[\sup_{r\in[s,t]} |X^x(r)-X^x(s)|^2\right] 
&\leq C\left(|t-s|^2 + \EE\left[\int_s^t\int_{\RR^n\setminus\{O\}}\left|F(X^x(r-),y) \right|^2\, \nu(dy)\,ds\right]\right)\\
&\leq C\left(|t-s|^2 + |t-s| \int_{\RR^n\setminus\{O\}} |\rho(y)|^2 \,\nu(dy)\right)\quad\hbox{ (by condition \eqref{eq:F_Levy_cond})}\\
&\leq C|t-s|\vee|t-s|^2.
\end{align*}
This concludes the proof of inequality \eqref{eq:Cont_X_2} and of Lemma~\ref{lem:Cont_X}. 
\end{proof}

\begin{proof}[Proof of Proposition~\ref{prop:Regularity_evol}]
We divide the proof into two steps.

\setcounter{step}{0}
\begin{step}[Lipschitz regularity in the spatial variable]
\label{step:Lipschitz_space}
Using the expression of the value function \eqref{eq:Value_function_evol}, we see that 
$|v(t,x_1)-v(t,x_2)| \leq \sup\{|v(t,x_1;\tau)-v(t,x_2;\tau)|:\tau\in\cT_{T-t}\}$, and so our goal is to prove that there is a positive constant $C$ 
such that
\begin{equation}
\label{eq:Lipschitz_space_aux}
|v(t,x_1;\tau)-v(t,x_2;\tau)| \leq C |x_1-x_2|,\quad\forall\, t\in[0,T],\quad\forall\, x_1,x_2\in\RR^n,\quad\forall\, \tau\in\cT_{T-t}.
\end{equation}
Similarly to the proof of Proposition~\ref{prop:Regularity_stat}, using the Lipschitz continuity of $c,f,\varphi$, and $g$, and the boundedness of the zeroth order term $c$ and of the stopping time $\tau \in [0,T]$, we obtain that
$$
|v(t,x_1;\tau)-v(t,x_2;\tau)| \leq C\EE\left[\max_{0\leq s\leq T-t} |X^{x_1}(s)-X^{x_2}(s)|\right].
$$
Applying H\"older's inequality and estimate \eqref{eq:Cont_X_1} to the right-hand side of the preceding inequality, we obtain that \eqref{eq:Lipschitz_space_aux} holds, and so we have that
\begin{equation}
\label{eq:Lipschitz_space}
\sup\{\|v(t,\cdot)\|_{C^{0,1}(\RR^n)}:\, t\in [0,T]\}<\infty.
\end{equation}
\end{step}

\begin{step}[$C^{\frac{1}{2}}$-regularity in the time variable]
\label{step:Holder_time}
We assume without loss of generality that $0\leq t_1<t_2\leq T$, and using the fact that $\cT_{T-t_2} \subset \cT_{T-t_1}$, we have that the following inequalities hold:
\begin{align*}
v(t_2,x) - v(t_1,x) &\leq \sup_{\tau\in\cT_{T-t_2}} \left(v(t_2,x;\tau) - v(t_1,x;\tau')\right),\\
v(t_1,x) - v(t_2,x) &\leq \sup_{\tau\in\cT_{T-t_1}}\left(v(t_1,x;\tau) - v(t_2,x;\tau'')\right),
\end{align*}
where we used the notation 
\begin{equation}
\label{eq:Stopping_time_tau}
\begin{aligned}
\tau':=\tau\mathbf{1}_{\{\tau<T-t_2\}} + (T-t_1)\mathbf{1}_{\{\tau=T-t_2\}}
\quad\hbox{ and }\quad
\tau'':=\tau\wedge (T-t_2)\in\cT_{T-t_2}.
\end{aligned}
\end{equation}
Our goal is to prove that there is a positive constant $C$ such that
\begin{align}
\label{eq:Lipschitz_time_1}
v(t_2,x;\tau) - v(t_1,x;\tau') &\leq C|t_1-t_2|,\quad\forall\, \tau\in\cT_{T-t_2},\\
\label{eq:Lipschitz_time_2}
v(t_1,x;\tau) - v(t_2,x;\tau'')&\leq C|t_1-t_2|^{\frac{1}{2}},\quad\forall\, \tau\in\cT_{T-t_1}.
\end{align}
The preceding four inequalities will imply that 
\begin{equation}
\label{eq:Holder_time}
\sup\{\|v(\cdot,x)\|_{C^{\frac{1}{2}}([0,T])}: x\in\RR^n\}<\infty.
\end{equation}
We split the proofs of inequalities \eqref{eq:Lipschitz_time_1} and \eqref{eq:Lipschitz_time_2} into two cases.
\setcounter{case}{0}
\begin{case}[Proof of inequality \eqref{eq:Lipschitz_time_1}]
\label{case:Lipschitz_time_1}
We denote for simplicity
$$
E(t,s,x):= e^{-\int_0^s c(t+r, X^x(r))\, dr},\quad\forall\, t\in[0,T],\, s\in[0,T-t],\, x\in\RR^n,
$$
and using the choice of the stopping times $\tau'$ in \eqref{eq:Stopping_time_tau}, we obtain the decomposition:
\begin{align*}
&v(t_2,x;\tau) - v(t_1,x;\tau') \\
&\qquad= 
\EE\left[\mathbf{1}_{\{\tau<T-t_2\}} \left(E(t_2,\tau,x) - E(t_1,\tau,x)\right)\varphi(t_2+\tau, X^x(\tau))\right]\\
&\qquad\quad+
\EE\left[\mathbf{1}_{\{\tau<T-t_2\}} E(t_1,\tau,x)\left(\varphi(t_2+\tau, X^x(\tau))- \varphi(t_1+\tau, X^x(\tau))\right)\right]\\
&\qquad\quad+
\EE\left[\mathbf{1}_{\{\tau = T-t_2\}} \left(E(t_2,\tau,x)-E(t_1,\tau,x)\right) g(X^x(\tau))\right]
\quad\hbox{($\tau'=T-t_1$ when $\tau=T-t_2$ by \eqref{eq:Stopping_time_tau})}\\
&\qquad\quad+
\EE\left[\int_0^{\tau} \left(E(t_2,r,x) - E(t_1,r,x)\right) f(t_2+r, X^x(r))\, dr\right]\\
&\qquad\quad + 
\EE\left[\int_0^{\tau} E(t_1,r,x)\left(f(t_2+r,X^x(r) - f(t_1+r, X^x(r))\right)\, dr\right]\\
&\qquad\quad -
\EE\left[\mathbf{1}_{\{\tau=T-t_2\}}\int_{T-t_2}^{T-t_1} E(t_1,r,x) f(t_1+r, X^x(r))\, dr\right].
\end{align*}
Inequality \eqref{eq:Lipschitz_time_1} is an immediate consequence of the preceding expression, the boundedness of the zeroth order term $c$, and of the Lipschitz continuity of the functions $c,f,\varphi$ and $g$.
\end{case}

\begin{case}[Proof of inequality \eqref{eq:Lipschitz_time_2}]
\label{case:Lipschitz_time_2}
Similarly to the proof of Case~\ref{case:Lipschitz_time_1}, by direct computation we can write the difference as a sum of three terms corresponding to the cases $\tau<T-t_2$, $T-t_2\leq\tau<T-t_1$, and $\tau=T-t_1$:
\begin{equation*}
v(t_1,x;\tau) - v(t_2,x;\tau'') = I_1 + I_2 + I_3,
\end{equation*}
where we denote:
\begin{align*}
I_1 &= 
\EE\left[\mathbf{1}_{\{\tau<T-t_2\}}\left( E(t_1,\tau,x)\varphi(t_1+\tau, X^x(\tau))-E(t_2,\tau,x)\varphi(t_2+\tau, X^x(\tau))\right)\right]\\*
&\quad +
\EE\left[\int_0^{\tau\wedge(T-t_2)} \left(E(t_1,r,x)f(t_1+r,X^x(r-) - E(t_2,r,x) f(t_2+r, X^x(r))\right)\, dr\right],\\
I_2 &=
\EE\left[\mathbf{1}_{\{T-t_2\leq \tau <T-t_1\}} \left(E(t_1,\tau,x)\varphi(t_1+\tau,X^x(\tau))-E(t_2,T-t_2,x)g(X^x(T-t_2))\right) \right]\\*
&\quad + 
\EE\left[\mathbf{1}_{\{T-t_2\leq \tau <T-t_1\}}\int_{T-t_2}^{\tau} E(t_1,r,x) f(t_1+r,X^x(r)\, dr\right],\\
I_3 & =
\EE\left[\mathbf{1}_{\{\tau=T-t_1\}} \left(E(t_1,\tau,x) g(X^x(T-t_1))-E(t_2,\tau,x) g(X^x(T-t_2))\right)\right].
\end{align*}
To prove inequality \eqref{eq:Lipschitz_time_2} we further expand each term as follows:
\begin{align*}
I_1 &= 
\EE\left[\mathbf{1}_{\{\tau<T-t_2\}} \left(E(t_1,\tau,x) - E(t_2,\tau,x)\right)\varphi(t_1+\tau, X^x(\tau))\right]\\*
&\quad+
\EE\left[\mathbf{1}_{\{\tau<T-t_2\}} E(t_2,\tau,x)\left(\varphi(t_1+\tau, X^x(\tau))-\varphi(t_2+\tau, X^x(\tau))\right)\right]\\*
&\quad+
\EE\left[\int_0^{\tau\wedge(T-t_2)} \left(E(t_1,r,x) - E(t_2,r,x)\right) f(t_2+r, X^x(r))\, dr\right]\\*
&\quad + 
\EE\left[\int_0^{\tau\wedge(T-t_2)} E(t_2,r,x)\left(f(t_1+r,X^x(r-) - f(t_2+r, X^x(r))\right)\, dr\right],\\
I_2 &=
\EE\left[\mathbf{1}_{\{T-t_2\leq \tau <T-t_1\}} \left(E(t_1,\tau,x)-E(t_2,\tau,x)\right)\varphi(t_1+\tau,X^x(\tau)) \right]\\*
&\quad+
\EE\left[\mathbf{1}_{\{T-t_2\leq \tau <T-t_1\}} \left(E(t_2,\tau,x)-E(t_2,T-t_2,x)\right)\varphi(t_1+\tau, X^x(\tau))\right]\\*
&\quad+
\EE\left[\mathbf{1}_{\{T-t_2\leq \tau <T-t_1\}} E(t_2,T-t_2,x)\left(\varphi(t_1+\tau, X^x(\tau))-\varphi(T,X^x(\tau))\right) \right]\\*
&\quad+
\EE\left[\mathbf{1}_{\{T-t_2\leq \tau <T-t_1\}} E(t_2,T-t_2,x)\left(\varphi(T,X^x(\tau))-g(X^x(\tau))\right) \right]\\*
&\quad+
\EE\left[\mathbf{1}_{\{T-t_2\leq \tau <T-t_1\}} E(t_2,T-t_2,x)\left(g(X^x(\tau))-g(X^x(T-t_2))\right) \right]\\*
&\quad + 
\EE\left[\mathbf{1}_{\{T-t_2\leq \tau <T-t_1\}}\int_{T-t_2}^{\tau} E(t_1,r,x) f(t_1+r,X^x(r)\, dr\right],\\
I_3 
& =
\EE\left[\mathbf{1}_{\{\tau=T-t_1\}} \left(E(t_1,\tau,x)-E(t_2,\tau,x)\right) g(X^x(T-t_1))\right]\\*
&\quad +
\EE\left[\mathbf{1}_{\{\tau=T-t_1\}} \left(E(t_2,\tau,x)-E(t_2,T-t_2,x)\right) g(X^x(T-t_1))\right]\\*
&\quad +
\EE\left[\mathbf{1}_{\{\tau=T-t_1\}} E(t_2,T-t_2,x)\left(g(X^x(T-t_1)) - g(X^x(T-t_2)) \right)\right].
\end{align*}
The compatibility condition \eqref{eq:Compatibility_evol} gives us that the fourth term in the expression of $I_2$ is non-positive.
Using the fact that $c$ is a bounded function and $g$ is Lipschitz continuous, we bound the fifth term in the expression of $I_2$ and the last term in the expression of $I_3$ by
\begin{align*}
&\left|\EE\left[\mathbf{1}_{\{T-t_2\leq \tau <T-t_1\}} E(t_2,T-t_2,x)\left(g(X^x(\tau))-g(X^x(T-t_2))\right) \right]\right|\\
&\quad + \left| \EE\left[\mathbf{1}_{\{\tau=T-t_1\}} E(t_2,T-t_2,x)\left(g(X^x(T-t_1)) - g(X^x(T-t_2)) \right)\right]\right|\\
&\qquad\leq C\EE\left[\sup_{T-t_2\leq s <T-t_1} \left|X^x(s)-X^x(T-t_2)\right| \right],
\end{align*}
for a positive constant $C$. Using again the fact that the zeroth order term $c$ is bounded and $c,f,\varphi$ and $g$ are Lipschitz functions, the remaining terms in the expression of the difference $v(t_1,x;\tau) - v(t_2,x;\tau'')$ can be bounded by $C|t_1-t_2|$, for a positive constant $C$. Hence, it follows that
\begin{align*}
v(t_1,x;\tau) - v(t_2,x;\tau'') 
\leq C|t_1-t_2| + C\EE\left[\sup_{T-t_2\leq s <T-t_1} \left|X^x(s)-X^x(T-t_2)\right| \right].
\end{align*}
Applying H\"older's inequality and estimate \eqref{eq:Cont_X_2} to the second term on the right-hand side, it follows that property \eqref{eq:Lipschitz_time_2} holds. 
\end{case}
The proofs of inequalities \eqref{eq:Lipschitz_time_1} and \eqref{eq:Lipschitz_time_2} imply estimate \eqref{eq:Holder_time}.
\end{step}
Properties \eqref{eq:Lipschitz_space} and \eqref{eq:Holder_time} complete the proof.
\end{proof}

We have the following analogue in the evolution case of the Dynamic Programming Principle in the stationary case proved in Lemma~\ref{lem:DPP_stat}:

\begin{lem}[Dynamic Programming Principle]
\label{lem:DPP_evol}
Suppose that the hypotheses of Proposition~\ref{prop:Regularity_evol} hold. Then, the value function $v(t,x)$ defined in \eqref{eq:Value_function_evol} satisfies:
\begin{equation}
\label{eq:DPP_evol}
v(t,x) = \sup\{v(t,x;r,\tau):\,\tau \leq \tau_r\wedge (T-t)\}, \quad\forall\, (t,x)\in[0,T)\times\RR^n,
\end{equation}
where we denote
\begin{equation}
\label{eq:DPP_evol_aux}
\begin{aligned}
v(t,x;r,\tau) 
&:=\EE\left[
e^{-\int_0^{\tau}c(t+s,X^x(s))\,ds} \varphi(t+\tau,X^x(\tau))\mathbf{1}_{\{\tau<\tau_r\wedge (T-t)\}}\right]\\
&\quad+\EE\left[ e^{-\int_0^{\tau}c(t+s,X^x(s))\,ds} v(t+\tau, X^x(\tau))\mathbf{1}_{\{\tau = \tau_r,\tau<T-t\}}\right]\\
&\quad+\EE\left[ e^{-\int_0^{\tau}c(t+s,X^x(s))\,ds} g(X^x(\tau))\mathbf{1}_{\{\tau < \tau_r,\tau=T-t\}}\right]\\
&\quad + \EE\left[\int_0^{\tau} e^{-\int_0^{\rho} c(t+s,X^x(s))\, ds} f(t+\rho,X^x(\rho))\, d\rho\right].
\end{aligned}
\end{equation}
\end{lem}

\begin{proof}
To prove Lemma~\ref{lem:DPP_evol} we can employ the argument used to establish Lemma~\ref{lem:DPP_stat} with the following changes. We choose the stopping times $\tau$ in $\cT_{T-t}$ instead of $\cT$, and we replace the use of the function $v(x;r,\tau)$ defined in \eqref{eq:DPP_stat_aux} by that of $v(t,x;r,\tau)$ in \eqref{eq:DPP_evol_aux}. In Step~\ref{step:Ineq_v_less_w} of the proof of Lemma~\ref{lem:DPP_stat} we condition on the $\sigma$-algebra $\cF_{\tau_r\wedge(T-t)}$ instead of $\cF_{\tau_r}$, and in Step~\ref{step:Ineq_w_less_v} of Lemma~\ref{lem:DPP_stat}, we choose the family of sets $\{A_k\}_{k\in\NN}$ such that it partitions $[0,T]\times\RR^n$ instead of $\RR^n$, and we replace the application of Proposition~\ref{prop:Regularity_stat} by that of Proposition~\ref{prop:Regularity_evol}. We omit the remaining details of the proof for brevity.
\end{proof}

We can now use Lemma~\ref{lem:DPP_evol} to prove the existence of viscosity solutions to the evolution obstacle problem \eqref{eq:Evolution_obs_problem}.

\begin{proof}[Proof of Theorem~\ref{thm:Existence_evol}]
The proof is very similar to that of Theorem~\ref{thm:Existence_stat} with the observation that we replace the application of Lemma~\ref{lem:DPP_stat} with that of Lemma~\ref{lem:DPP_evol}. We omit the detailed proof.
\end{proof}

\begin{thm}[Comparison principle]
\label{thm:Comparison_princ_evol}
Suppose that Assumption~\ref{assump:Coeff} is satisfied, $g$ belongs to $C(\RR^n)$, $c,f,\varphi$ are in $C([0,T]\times\RR^n)$, and conditions \eqref{eq:Compatibility_evol} and \eqref{eq:Interior_point} hold. If $u$ and $v$ are a viscosity subsolution and supersolution to the evolution obstacle problem \eqref{eq:Evolution_obs_problem}, respectively, then $u\leq v$.
\end{thm}

\begin{proof}
The proof is identical to that of Theorem~\ref{thm:Comparison_princ_stat} except for the following modifications. We replace the definition of $M_{\alpha,\eps}$ in \eqref{eq:M_a_e_stat} by
\begin{equation}
\label{eq:M_a_e_evol}
M_{\alpha,\eps} := \sup_{(t,x),(s,y)\in[0,T]\times\RR^n} \left\{u(t,x)-v(s,y) - \frac{\alpha}{2}\left(|x-y|^2+|t-s|^2\right)-\eps(|x|^2+|y|^2)\right\}.
\end{equation}
In the proof of Theorem~\ref{thm:Comparison_princ_stat} we used the fact that the zeroth order coefficient $c(x)$ is positive, which is not necessarily true under the hypotheses of Theorem~\ref{thm:Comparison_princ_evol}. We notice that because $c(x)$ is bounded, there is a positive constant $\lambda$ such that $\lambda+c(x)>0$. From Definition~\ref{defn:Viscosity_sol_evol} it is clear that if $v$ is a viscosity sub- or super-solution to the evolution obstacle problem \eqref{eq:Evolution_obs_problem}, then $e^{\lambda(T-t)}v$ is a viscosity sub- or supersolution to equation \eqref{eq:Evolution_obs_problem}, where we replace $c(x)$ by $c(x)+\lambda>0$, $f(t,x)$ by $e^{\lambda(T-t)}f(t,x)$, and $\varphi(t,x)$ by $e^{\lambda(T-t)}\varphi(t,x)$.
We omit the remainder of the proof for brevity.
\end{proof}

\begin{proof}[Proof of Theorem~\ref{thm:Uniqueness_evol}]
The conclusion of Theorem~\ref{thm:Uniqueness_evol} is an immediate consequence of Theorem~\ref{thm:Comparison_princ_evol}.
\end{proof}

%
%

\bibliography{mfpde}
\bibliographystyle{amsplain}

\end{document}